\documentclass[preprint,10pt]{imsart}

\usepackage[square,numbers]{natbib}
\RequirePackage[colorlinks,citecolor=blue,urlcolor=blue]{hyperref}
\usepackage{mathrsfs,amsthm,amsmath}
\usepackage{amssymb,multirow}
\usepackage{amsfonts}
\usepackage{stmaryrd}

\usepackage{dsfont}
\usepackage{ucs}
\usepackage[utf8x]{inputenc}
\usepackage{amsmath}
\usepackage{amsfonts}
\usepackage{amssymb}
\usepackage{amsthm}
\usepackage{fontenc}
\usepackage{graphicx}

\usepackage{bbm}
\usepackage{bm}
\RequirePackage[OT1]{fontenc}

\startlocaldefs

\newcommand{\Ex}[1]{\ensuremath{\mathds{E} \left[#1 \right]}}
\newcommand{\Prob}[1]{\ensuremath{\mathds{P} \left(#1 \right)}}

\newcommand{\mbb}[1]{\mathbb{#1}}

\newcommand{\wt}[1]{\widetilde{#1}}
\newcommand{\wh}[1]{\widehat{#1}}
\newcommand{\Xc}{\mathcal{X}}

\newcommand{\Ac}{\mathcal{A}}
\newcommand{\ep}{\varepsilon}

\newcommand{\indic}{\mathds{1}}

\newcommand{\X}{\mathbb{X}}
\newcommand{\Z}{\mathbb{Z}}

\newcommand{\R}{\mathds{R}}
\newcommand{\E}{\mathds{E}}
\newcommand{\N}{\mathds{N}}
\renewcommand{\P}{\mathds{P}}

\newcommand{\Bcr}{\mathscr{B}}
\newcommand{\Fcr}{\mathscr{F}}

\newcommand{\Pcr}{\mathscr{P}}
\newcommand{\Xcr}{\mathscr{X}}
\newcommand{\Mcr}{\mathscr{M}}

\newcommand{\ddr}{\mathrm{d}}
\newcommand{\edr}{\mathrm{e}}

\def\eqd{\stackrel{\mbox{\scriptsize{d}}}{=}}
\theoremstyle{plain}
\newtheorem{thm}{\textsc{Theorem}}
\newtheorem{defi}{\textsc{Definition}}
\newtheorem{lem}{\textsc{Lemma}}
\newtheorem{cor}{\textsc{Corollary}}
\newtheorem{rmk}{Remark}

\newtheorem{exe}{\textsc{Example}}
\allowdisplaybreaks
\endlocaldefs

\begin{document}
\begin{frontmatter}

\title{Canonical correlations for dependent gamma processes }
\runtitle{Canonical correlations for dependent gamma processes}

\author{\fnms{Dario Span\`o}\thanksref{t1}\ead[label=e1]{D.Spano@warwick.ac.uk}}
\and
\author{\fnms{Antonio Lijoi}\thanksref{t2}\ead[label=e2]{lijoi@unipv.it}}

\address{Dario Span\`o   \\
Department of Statistics \\
University of Warwick    \\
Coventry CV4 7ALv\\
United Kingdom \\
 \printead{e1}}

\address{Antonio Lijoi\\ Department of Economics and Management\\ University of Pavia\\ Via San
Felice 5, 27100 Pavia, Italy \\ \printead{e2}}

\runauthor{D. Span\`o and A. Lijoi}
\thankstext{t1}{Supported by CRiSM, an EPSRC-HEFCE UK grant.}
\thankstext{t2}{Also affiliated to Collegio Carlo Alberto, Moncalieri, Italy.
Supported by MIUR, grant 2008MK3AFZ.}

\affiliation{University of Warwick and Universit\`a di Pavia}

\vspace{0.4cm}
\begin{abstract}
The present paper provides a characterisation of exchangeable pairs of random measures $(\wt\mu_1,\wt\mu_2)$ whose identical margins are fixed to coincide
with the distribution of a gamma completely random measure,
and whose dependence structure is given in terms of canonical correlations.
It is first shown that canonical correlation sequences for the finite-dimensional distributions of $(\wt\mu_1,\wt\mu_2)$ are moments of means of a Dirichlet process having  random base measure.   
Necessary and sufficient conditions are further given for canonically correlated  gamma completely random measures to have independent joint increments. Finally, time-homogeneous Feller processes with gamma reversible measure and canonical autocorrelations are characterised as Dawson--Watanabe diffusions with independent homogeneous immigration, time-changed via an independent subordinator. It is thus shown that Dawson--Watanabe diffusions subordinated by pure drift are the only processes in this class whose time-finite-dimensional distributions have, jointly, independent increments.

\end{abstract}



\end{frontmatter}

\section{Introduction} The definition of probability distributions on $\R^d$, with fixed margins, has a long history. An approach that stands out for its elegance, and the wealth of mathematical results it yields, is due to H.O.~Lancaster (see \cite{Lanc(58)}) who used orthogonal functions on the marginal distributions. This can be summarised in a formulation that is strictly related to the purpose of the present paper as follows. Suppose $\pi_1$ and $\pi_2$ are two probability measures on $\R^d$ and let $\{Q_{1,\bm{n}}:\: \bm{n}\in\Z_+^d\}$ and $\{Q_{2,\bm{n}}:\: \bm{n}\in\Z_+^d\}_{n=1}^\infty$ be collections of multivariate functions that form complete orthonormal sets on $\pi_1$ and $\pi_2$, respectively. Moreover, let $\bm{X}$ and $\bm{Y}$ be random vectors defined on some probability space $(\Omega,\Fcr,\P)$ and taking values in $\R^d$ with probability distribution $\pi_1$ and $\pi_2$, respectively. If for every $\bm{m}$ and $\bm{n}$ in $\Z_+^d$ one has
\begin{equation}
\E\left[ Q_{1,\bm{n}}(\bm{X})Q_{2,\bm{m}}(\bm{Y})\right]=\rho_{\bm{n}}\,\delta_{\bm{n},\bm{m}},
\label{ccdef1}
\end{equation}
with $\delta_{\bm{n},\bm{m}}=1$ if $\bm{n}=\bm{m}$ and $0$ otherwise, then $\bm{X}$ and $\bm{Y}$ are said to be canonically correlated, and $\{\rho_{\bm{n}}:\: \bm{n}\in\Z_+^d\}$ is called the \emph{canonical correlation sequence} of $(\bm{X},\bm{Y})$. 
For $d=1$ a complete characterisation of the set of all possible canonical correlation sequences for bivariate distributions with gamma marginals is given in \cite{Bobbo(69)}. In this case the Laguerre polynomials identify a family of orthogonal functions. 

In the present paper we aim at characterising canonical correlations in infinite dimensions for pairs of random measures with identical marginal distributions given by the law of a gamma completely random measure. To this end, we first need to introduce some notation and point out a few definitions. Let $\Xc$ be a complete and separable metric space and $\Xcr$ the Borel $\sigma$-algebra on $\Xc$. Set $M_\Xc$ as the space of boundedly finite measures on $(\Xc,\Xcr)$ equipped with a Borel $\sigma$--algebra $\Mcr_\Xc$ (see \cite{daley} for details). A \textit{completely random measure} (CRM) $\wt\mu$ is a measurable mapping from $(\Omega,\Fcr,\P)$ into $(M_\Xc,\Mcr_\Xc)$ such that, for any finite collection $\{A_1,\ldots, A_d\}$ of disjoint sets in $\Xcr$ and any $d\in\N$, the random variables $\wt\mu(A_i)$ are independent, for $i=1,\ldots,d$. In particular, $\wt \mu$ is a \textit{gamma} CRM with parameter measure $cP_0$ if
\[
\phi_{(c,P_0)}(f):=\E\left[\edr^{-\wt\mu(f)}\right]=\exp\left\{-c\,\int_\Xc \log(1+f(x))\, P_0(\ddr x)\right\}
\]
for any measurable function $f:\Xc\to\R$ such that $\int \log(1+|f|)\,\ddr P_0<\infty$, where $c>0$ and $P_0$ is a probability measure on $\Xc$. Henceforth we will use the notation $\Gamma_{cP_0}$ when referring to it. Furthermore, $P_0$ is assumed non--atomic.

 {In section 4 we will define a pair of random measures $(\wt\mu_1,\wt\mu_2)$ to be in canonical correlation if all its finite-dimensional distributions have canonical correlations. The distributional properties of any such pair, when each $\wt\mu_i$ is a gamma CRM with identical parameter, will be investigated by resorting to the Laplace functional transform}
\begin{equation}
  \label{eq:biv_laplace}
  \phi(f_1,f_2)=\E\left[\edr^{-\wt\mu_1(f_1)-\wt\mu_2(f_2)}\right]
\end{equation}
for any pair of measurable functions $f_i:\Xc\to\R$ such that $\int\log(1+|f_i|)\,\ddr P_0<\infty$, for $i=1,2$. It will be shown that a vector $(\wt{\mu}_1,\wt{\mu}_2)$ of gamma CRMs has canonical correlations if and only if its joint Laplace functional has the following representation:
\begin{equation}
  \label{eq:repres_biv_lapl}
  \frac{\phi(f_1,f_2)}{\phi_{(c,P_0)}(f_1)\,\phi_{(c,P_0)}(f_2)}=
  \E\left[\exp\left\{\int_\Xc \frac{f_1f_2}{(1+f_1)(1+f_2)}\:\ddr K\right\}\right]
\end{equation}
for some random measure $K$ which can be appropriately expressed as a linear functional of an independent, gamma random measure $\wt{\mu}$ with random intensity.  This is our main Theorem \ref{thm:main}. 
We will also find out a surprising fact: the canonical correlations associated to the finite--dimensional distributions of $(\wt\mu_1,\wt\mu_2)$ can be represented as mixed moments of linear functionals of Dirichlet processes. This follows from the combination of Theorems \ref{thm:zeta} and \ref{thm:main} 
below and connects our work to various, and seemingly unrelated, areas of research where means of Dirichlet processes play an important role. See \cite{LPsurvmeans} far a detailed account. Using orthogonal functions expansions, along with the interpretation as moments of Dirichlet random means, simplifies the analytic treatment of canonically correlated gamma measures, and lead to easy algorithms for sampling from their joint distribution. We will describe such algorithms and show that they can all be thought of as generalisations of a neat Gibbs sampling scheme proposed in \cite{DKSC08} for a class of one-dimensional gamma canonical correlations.
This analysis also allows us to identify the form of canonical correlations that yield a vector $(\wt\mu_1,\wt\mu_2)$ which is a CRM vector itself. 

Our investigation is, finally, extended to encompass a characterisation of canonical (auto)correlations for continuous time, time-homogeneous and reversible measure-valued Markov processes with gamma stationary distribution. We will use, once again, the connection with Dirichlet random means to prove that all the Markov processes in this class can be derived via a Bochner-type subordination of a well-known measure-valued branching diffusion process with immigration (its generator is recalled in Section \ref{sec:illustr}, (\ref{eq:MBDIgen})), to which we will refer as the $\Gamma$-Dawson-Watanabe process. We also prove that the  $\Gamma$-Dawson-Watanabe process is, up to a deterministic rescaling of time, the only instance in this class with the property that all its time-bidimensional distributions are, jointly, completely random measures.
We also use our algorithms and a result of \cite{EG93} on $\Gamma$-Dawson-Watanabe process, to derive a simple expansion for the transition function of general stationary gamma processes with canonical autocorrelations.

The outline of the paper is as follows. In Section~2 we recall some background on canonical correlations for random variables and specialise the discussion to the gamma case. In Section~3 we move on to considering canonically correlated vectors with independent gamma random components. 
Algorithms for simulating such vectors are also described. In Section~4 we prove the characterisation \eqref{eq:repres_biv_lapl} of canonically correlated Gamma CRMs in terms of their joint Laplace functional and in terms of moments of Dirichlet random means. Section~5 discusses several notable examples corresponding to different specifications of the canonical correlations or, equivalently, of the driving measure $K$.
Finally, Section~6 characterises time--homogenous and reversible Markov processes with values in $M_\Xc$ with canonical autocorrelations.

\section{Canonical correlations for gamma random variables}\label{sec:1-dim} Canonical correlations were originally defined via orthogonal polynomials, and firstly applied in Statistics in \cite{Lanc(58)}. Let $\pi$ be a probability measure on $\R^d$ and $\{Q_{\bm{n}}:\: \bm{n}\in\Z^d_+\}$ a complete system of orthonormal polynomials with weight measure $\pi$ indexed by their degree $\bm{n}\in\Z^d_+$. The \emph{total degree} of $Q_{\bm{n}}$ will be denoted by $|\bm{n}|=n_1+\cdots+n_d$. Hence $\int Q_{\bm{n}} \, Q_{\bm{m}}\,\ddr\pi=\delta_{\bm{n},\bm{m}}$.
It can be easily shown that any two random vectors $\bm{X}$ and $\bm{Y}$, sharing the same marginal distribution $\pi$, are canonically correlated as in \eqref{ccdef1}, with canonical correlation sequence $\{\rho_{\bm{n}}:\bm{n}\in\Z_+^d\},$ if and only if
\begin{equation}
\E\left[Q_{\bm{n}}(\bm{Y}) \mid \bm{X}=\bm{x}\right]=\rho_{\bm{n}}\, Q_{\bm{n}}(\bm{x})\ \ \ \bm{n}\in\Z_{+}^d,
\label{eq:cexp}
\end{equation}
for every $\bm{y}$ in the support of $\pi$.
Under this condition, since any function $f$ such that $\int_{\R^d} f^2\: \ddr\pi<\infty$ has a series representation in $L^2(\R^d,\pi)$:
$$f(\bm{x})=\sum_{\bm{n}\in\Z_+^d}^\infty \wh{f}(\bm{n})Q_{\bm{n}}(\bm{x})\label{l2series}
$$
where $\wh{f}(\bm{n}):=\Ex{f(\bm{X})Q_{\bm{n}}(\bm{X})}$, for any $\bm{n}\in\Z_+^d ,$ then one has
$$\Ex{f(\bm{Y})\mid \bm{X}=\bm{x}}=\sum_{\bm{n}\in\Z_+^d} \wh{f}(\bm{n})\rho_{\bm{n}}Q_{\bm{n}}(\bm{x}).$$
Note that $\rho_{\bm{0}}$ must be equal to 1 in order, for the conditional expectation operator, to map constant functions to constant functions.
The canonical correlation coefficients $\rho$ thus contain all the information about the dependence between $\bm{X}$ and $\bm{Y}$. At the extremes, \emph{\lq\lq $\rho_{\bm{n}}=0$ for every $\bm{n}\in\Z_+^d\setminus\{\bm{0}\}$\rq\rq}\ implies independence and \emph{\lq\lq$\rho_{\bm{n}}=1$ for every ${\bm{n}}\in\Z_+^d$\rq\rq}\ corresponds to perfect dependence. Expansions in $L^2$ can be easily determined for the joint and conditional Laplace transforms as soon as an appropriate generating function for the orthogonal polynomials is available, as in the Gamma case.

We now focus on the case where $d=1$ and $\pi$ is a gamma distribution, namely
\[
\pi(\ddr x)=\Gamma_{\alpha,\beta}(\ddr x)=\frac{1}{\beta^\alpha\Gamma(\alpha)}\, x^{\alpha-1}\:\edr^{-\frac{x}{\beta}}\,\ddr x\: \indic_{(0,\infty)}(x)
\]
where $\alpha>0$ and $\beta>0$ and $\indic_A$ is the indicator function of $A$.
It is worth recalling a few well--known facts about gamma random variables in canonical correlation, mostly owed to several works of Eagleson and Griffiths (see for example  \cite{E64, Bobbo(69),Bobbo(70)}).

The Laplace transform of a $\Gamma_{\alpha,\beta}$ distribution is denoted as
$\phi_{\alpha,\beta}(t)=\left(1+\beta t\right)^{-\alpha}.
$
With no loss of generality in the sequel we maintain the assumption $\beta=1$, unless otherwise stated.

A complete system of orthogonal polynomials with respect to $\Gamma_{\alpha,1}$ is represented by the Laguerre polynomials
\[
L_{n,\alpha}(x)=\frac{(\alpha)_{n}}{n!}\: _1F_1(-n;\alpha;x)
\]
for $\alpha>0$, where $_1F_1$ is the confluent hypergeometric function
 $$_1F_1(a;b;z)=\sum_{n=0}^\infty\frac{(a)_n}{(b)_n}\frac{z^n}{n!}$$
 and $(\alpha)_n=\Gamma(\alpha+n)/\Gamma(\alpha)$ is the $n$--th ascending factorial of $\alpha$. It can be easily checked that
\[
\int_{\R_+} L_{n,\alpha}(x)L_{m,\alpha}(x)\: \Gamma_{\alpha,1}(\ddr x)=\frac{(\alpha)_n}{n!}\:\delta_{n,m}.
\]
Hence, a collection of orthonormal polynomials $\{L_{n,\alpha}^*:\: n\ge 1\}$ can be defined by setting $L_{n,\alpha}^*=(n!/(\alpha)_n)^{1/2}\,L_{n,\alpha}$. Finally, Laguerre polynomials with leading coefficient equal to unity will be henceforth denoted as $\wt L_{n,\alpha}$, where $\wt L_{n,\alpha}=n!(-1)^n\, L_{n,\alpha}$. From this definition one finds out that
\begin{equation}
  \label{eq:orthog_laguerre}
  c_{n,\alpha}:=\E\left[\wt{L}_{n,\alpha}(X)^2 \right]=n!\,(\alpha)_n.
\end{equation}
Under this scaling, closed form expressions can be determined for the generating functions
\begin{equation}
G_{\alpha}(r,x):=\sum_{n=0}^\infty\wt{L}_{n,\alpha}(x)\frac{r^n}{n!}
=(1+r)^{-\alpha}\edr^{\frac{xr}{1+r}}\ \ \ |r|<1\label{Laggf}
\end{equation}
and the Laplace transform
\begin{equation}
\psi_{n,\alpha}(t):=\int_0^\infty \edr^{-t x}\wt{L}_{n,\alpha}(x)\Gamma_{\alpha,1}(\ddr x)
=(\alpha)_{n}\left(\frac{-t}{t+1}\right)^n(1+t)^{-\alpha}.\label{Laglap}
\end{equation}
See \cite{E64}. 
An important result we refer to was proved in \cite{Bobbo(69)} and is as follows.

\begin{thm}[Griffiths \cite{Bobbo(69)}]\label{th:G69}
A sequence $(\rho_n)_{n\ge 0}$ is a sequence of canonical correlation coefficients of the law of a pair $(X,Y)$ with identical marginals  $\Gamma_{\alpha,1}$  if and only if $\rho_n=\E[Z^n]$, for a random variable $Z$ whose distribution has support $[0,1]$.
\end{thm}

\noindent Theorem \ref{th:G69} implies that the set of all gamma canonical correlation sequences is convex and its extreme points are in the set $\left\{(z^n)_{n\ge 0}:\: z\in[0,1]\right\}$. For any $z\in[0,1]$, set $\E_z$ as the expected value computed with respect to the probability distribution of $(X,Y)$ associated to the ``extreme'' canonical correlation sequence $(z^n)_{n\ge 0}$. Then
\begin{align}
  \phi_z(s,t)
  &=
  \E_z\left[\edr^{-sX-tY}\right]\notag
  ={(1+s)^{-\alpha}(1+t)^{-\alpha}\sum_{n\ge 0}\frac{(\alpha)_n}{n!}
  \, z^n\,\theta^n}\notag\\
  &=\phi_{\alpha,1}(s)\,\phi_{\alpha,1}(t)\,\phi_{\alpha,z}(-\theta)
\label{extremejlap}\end{align}
where $\theta=st(1+s)^{-1}(1+t)^{-1}.$ The corresponding density is well-known (see e.g. \cite{AA}, formula (6.2.25), p. 288) to be
\begin{equation}
p_z(\ddr x,\ddr y)=\Gamma_{\alpha,1}(\ddr x)\Gamma_{\alpha,1}(\ddr y)\frac{\edr^{-\frac{(x+y)z}{1-z}}}{(xyz)^{\frac{\alpha-1}{2}}}I_{\alpha-1}\left(\frac{2(xyz)^{\frac{1}{2}}}{1-z}\right),
\label{analyt}
\end{equation}
where $I_{c}$ is the modified Bessel function of order $c$, $c>-1$. If we allow the re-parametrisation $z=\edr^{-t}$ (for some $t\geq 0$), we recognise in (\ref{analyt}) the well-known joint distribution of two coordinates, taken a time $t$ apart from each other, of a one-dimensional continuous state branching process with homogeneous immigration (CSBI), also known, in the literature of mathematical finance, as the Cox-Ingersol-Ross (CIR) process, whose generator, under an appropriate scaling of the parameters, is given by
\begin{equation}
Lf=\frac{1}{2}x\frac{\partial^2}{\partial x^2}f+\frac{1}{2}(\alpha - x)\frac{\partial}{\partial x}f\label{1csbi},
\end{equation} for $f\in C^2_c([0,\infty)).$ The measure-valued extension of this process will play a central role in the second part of this paper (Sections 5.1 and \ref{sec:markovcc}). One can use (\ref{Laglap}) and (\ref{Laggf}) to deduce a simple form for the conditional Laplace transform of $Y$ given $X$ in the extreme family, namely
\begin{align}
\E_z\left[\edr^{-sY}\,\mid\, X=x\right]
&=
(1+s)^{-\alpha}G_{\alpha}\left(\frac{-sz}{1+s},x\right)\notag\\
&=\phi_{\alpha,1}(s(1-z))\,
\edr^{-x\frac{sz}{1+s(1-z)}}.\label{phizx}
\end{align}
Thus, for general gamma canonical correlation sequences $\rho=(\rho_n)_{n\ge 0}$, by Theorem~\ref{th:G69} the corresponding joint Laplace transform is a mixture of extreme Laplace transforms $\phi_\rho(s,t)=\Ex{\phi_Z(s,t)}$,
where $Z$ is the random variable defining $\rho$ via $\rho_n=\Ex{Z^n}.$

A simple algorithm is available to simulate from the joint distribution of a pair $(X,Y)$ of $\Gamma_{\alpha,1}$ random variables with canonical correlation in the extreme family (i.e. of the form $\rho_n=z^n$, $n\in\Z_+$, $z\in(0,1)$). After re-scaling, the algorithm reduces to a version of a Gibbs-sampling scheme (the so-called Poisson/Gamma $\theta$-chain) proposed by Diaconis, Khare and Saloff-Coste (\cite{DKSC08}, Proposition 4.6) to build one-step Markov kernels with $\Gamma_{\alpha,\beta}$ stationary measure and orthogonal polynomial eigenfunctions.

\medskip

 \textsc{Algorithm A.1}
\begin{enumerate}
\item[\texttt{(0)}] Initialise by choosing $\alpha>0$ and $b\in[0,\infty)$
\item[\texttt{(i)}] Generate a sample $X=x$ from the $\Gamma_{\alpha,1}$ measure
\item[\texttt{(ii)}] Generate $N=n$ from a Poisson distribution with mean $bx$
\item[\texttt{(iii)}] Generate ${Y}$ from $\Gamma_{\alpha+n,(1+b)^{-1}}$
\end{enumerate}


We thus have:
\begin{lem}\label{prp:A1}
The pair $(X,Y)$ generated by Algorithm A.1 is {exchangeable}, with identical $\Gamma_{\alpha,1}$ marginal distributions and canonical correlation coefficients of the form $\rho_n=z^n$, $n\in\Z_+$ where $
z=b/(1+b)$.
Conversely, every bivariate gamma pair with canonical correlations of the form $\rho_n=z^n$, $n\in\Z_+,z\in(0,1),$ can be generated via Algorithm A.1 with $b=z/(1-z).$ \end{lem}

\begin{proof}
From the algorithm, one can see that, for every $x\in[0,\infty),$ the conditional distribution of $Y$ given $X$ has the representation
\begin{equation}
p_b(x,dy)=\sum_{n=0}^\infty\frac{(bx)^n\edr^{-bx}}{n!} \Gamma_{\alpha+n,(1+b)^{-1}}(dy).\label{ctxtr}
\end{equation}
Thus the conditional Laplace transform is:
\begin{eqnarray}
\Ex{\edr^{-tY}\mid X=x}&=&\sum_{n=0}^{\infty}\frac{(bx)^n\edr^{-bx}}{n!}\phi_{\alpha+n,(1+b)^{-1}}(t)\notag\\
 &=&\left(1+\frac{t}{1+b}\right)^{-\alpha}\sum_{n=0}^{\infty}\frac{(bx)^n\edr^{-bx}}{n!}\left(\frac{1+b}{1+b+t}\right)^{n}\notag\\
 &=&\left(1+\frac{t}{1+b}\right)^{-\alpha}\exp\left\{bx\left(\frac{1+b}{1+b+t}-1\right)\right\}\notag\\
 &=&\left(1+{t}(1-z)\right)^{-\alpha}\edr^{-\frac{xtz}{1+t(1-z)}},\label{pgfmgf}
 \end{eqnarray}
where the second equality comes from the form of the gamma Laplace transform, the third from the form of the Poisson probability generating function, and the last one by substituting $z=b/(1+b).$ The right-hand side of (\ref{pgfmgf}) is equal to (\ref{phizx}) and that suffices to prove both the necessity and the sufficiency part of the Lemma.
\end{proof}

If we denote $U=bX$ and $V=bY,$ where $(X,Y)$ is the pair generated by Algorithm A.1, then $(U,V)$ is a bivariate pair with identical $\Gamma_{\alpha,b}$ marginal laws. Furthermore, the distribution of $V$ conditional on $U=u$
coincides with the one-step transition distribution of Diaconis \emph{et al.}'s $\theta$-chain. See \cite{DKSC08}, formula (4.7), and compare it to (\ref{ctxtr}). Thus one can think of canonical correlations in the extreme family 
as the one-step correlations obtained by re-scaling the values of the $\theta$-chain by a factor of $(1-z)/z$. Lemma~\ref{prp:A1} also shows that Algorithm A.1 (with $b=\edr^{-t}/(1-\edr^{-t})$) is all one needs to simulate from the bivariate density of a CSBI/CIR process with generator (\ref{1csbi}).\\
In the light of Theorem \ref{th:G69}, algorithm A.1 can be extended in a way that it incorporates a randomisation of $b$: the resulting variable, denoted as $B$, has some probability distribution $P^*$ on $\R_+$ and is independent of $X$ and $Y$.  One, then, runs algorithm A.1 conditional on $B=b$ and the output will still be a set of realisations of a canonically correlated vector $(X,Y)$ with $\Gamma_{\alpha,1}$ marginals. This is the key idea at the basis of Algorithm A.2 proposed in the next section to generate general canonically correlated gamma vectors in $d$ dimensions.
%
%

\section{Canonical correlations and algorithms for gamma product measures}
In the present Section we derive a characterisation of canonical correlations for pairs of random vectors with independent gamma coordinates, extending Griffiths' result recalled in Theorem~\ref{th:G69}. It is a preliminary step for stating the main result that will be discussed in the next Section. We shall focus on $d$--dimensional vectors $\bm{X}$ and $\bm{Y}$ having the same marginal distribution, $\pi_1=\pi_2=\Gamma_{\bm{\alpha},1}:=\times_{i=1}^d \Gamma_{\alpha_i,1}$. Since it is always possible to generate orthogonal polynomials for $d$-fold product measures as products of the $d$ corresponding \lq\lq marginal\rq\rq orthogonal polynomials, canonical correlations are then determined by \eqref{ccdef1} with
\[
Q_{1,\bm{n}}(\bm{x})=Q_{2,\bm{n}}(\bm{x})=\prod_{i=1}^d \frac{\wt{L}_{n_i,\alpha_i}(x_i)}{\sqrt{c_{n_i,\alpha_i}}}
\]
where $\wt{L}_{n_i,\alpha_i}$ stands for the Laguerre polynomial of degree $n_i$ with leading coefficient of unity and constant of orthogonality $c_{n_i,\alpha_i}$ as described in Section \ref{sec:1-dim}, and we will henceforth use the short notation $\wt{L}_{\bm{n},\bm{\alpha}}(\bm{x})=\prod_{i=1}^d \wt{L}_{n_i,\alpha_i}(x_i)$ for any $\bm{x}=(x_1,\ldots,x_d)\in\R^d$, $\bm{n}=(n_1,\ldots,n_d)\in\N^d$ and $\bm{\alpha}=(\alpha_1,\ldots,\alpha_d)\in (0,\infty)^d$.
The leading coefficient is equal to one and $c_{\bm{n},\bm{\alpha}}:=\E[\wt{L}_{\bm{n},\bm{\alpha}}^2(\bm{X})]=\prod_{i=1}^d n_i!\,(\alpha_i)_{n_i}$. A key for generating canonically correlated gamma random vectors is provided by the following algorithm.

\medskip

\textsc{Algorithm A.2}
\begin{enumerate}
\item[\texttt{(0)}] Initialise by setting $\bm{\alpha}\in\R_+^d$ and a probability distribution $P^*$ on $\R_+^d$
\item[\texttt{(i)}] Generate a sample $\bm{B}=\bm{b}$ from $P^*$
\item[\texttt{(ii)}] Given $\bm{b}$, generate $d$ conditionally independent runs of Algorithm A.1 whereby, for every $j=1,\ldots,d$, the $j$-th run is initialisiseed by $(\alpha_j,b_j).$ \item[\texttt{(iii)}] Return $X_j,Y_j$  for any $j=1,\ldots,d$.
\end{enumerate}
\medskip

According to the next statement, the above algorithm does indeed characterise canonically correlated gamma random vectors and further provides a description of the canonical correlations.

\begin{thm}
\label{thm:dprodcc} Let $(\bm{X},\bm{Y})$ be a pair of random vectors in $\R_+^d$ with identical marginal distribution $\Gamma_{\bm{\alpha},1}$, with $\bm{\alpha}\in \R_+^d.$ Then  $(\bm{X},\bm{Y})$ has canonical correlations if and only if it can be generated via {\rm Algorithm~A.2} for some probability measure $P^*$ on $\R_+^d.$ The canonical correlation coefficients are of the form
\begin{equation}\label{eq:can.corr.mult}
  \rho_{\bm{n}}=\E\left[\prod_{i=1}^d Z_i^{n_i}\right]
\end{equation}
where $Z_i=B_i/(1+B_i),$ $i=1,\ldots,d,$ and the vector $\bm{B}=(B_1,\ldots,B_d)$ has distribution $P^*$.
\end{thm}

\begin{proof}
Conditioning on $\bm{B}=(b_1,\ldots,b_d)$ the distribution is the same one would obtain by running $d$ independent copies of algorithm A.1, each $i$-th copy being initialised by $\alpha_i,b_i$, $i=1,\ldots,d.$ Then from Lemma 1
\begin{eqnarray*}\Ex{\wt{L}_
{\bm{n},\bm{\alpha}}\left(\bm{X}\right)\wt{L}_
{\bm{m},\bm{\alpha}}\left(\bm{Y}\right)}&=&\Ex{\Ex{\wt{L}_
{\bm{n},\bm{\alpha}}\left(\bm{X}\right)\wt{L}_
{\bm{m},\bm{\alpha}}\left(\bm{Y}\right)\mid \bm{B}}}\\
&=&\Ex{\prod_{i=1}^dc_{n_i,\alpha_i}\delta_{n_im_i}\left(\frac{B_i}{1+B_i}\right)^{n_i}}\\
&=&c_{\bm{n},\bm{\alpha}}\delta_{\bm{n}\bm{m}}\Ex{\prod_{i=1}^dZ_i^{n_i}}
\end{eqnarray*}
where $Z_i:=B_i/(1+B_i),\ i=1,\ldots,d$ and the sufficiency part follows.
\\
 For the necessity, we use an approach similar to one used in \cite{GM78} to construct bivariate Poisson vectors. Let us suppose the pair $(\bm{X},\bm{Y})$ have canonical correlation sequence $\rho(\bm{\alpha})=\{\rho_{\bm{n}}(\bm{\alpha})\}$. Let $f\in L^2(\times_{i=1}^d \Gamma_{\alpha_i,1})$ and set $$\hat{f}(\bm{n}):=\E_{\bm{\alpha}}\left[f(\bm{X})\wt{L}_{\bm{n},\bm{\alpha}}(\bm{X}) \right].$$ From (\ref{eq:cexp}),
 $$\Ex{f(\bm{Y})\mid \bm{X}=\bm{x}}=\sum_{\bm{n}\in\Z_+^d}\frac{\hat{f}(\bm{n})}{c_{\bm{n},\bm{\alpha}}}\rho_{\bm{n}}(\bm{\alpha})\wt{L}_{\bm{n},\bm{\alpha}}(\bm{x}).$$
Then
for every $\bm{s}\in\R_+^d,$
 \begin{align}
\Ex{\edr^{-\langle\bm{s},\bm{Y}\rangle}\mid \bm{X}=\bm{x}}
&=\sum_{\bm{n}\in\Z_+^d}
\rho_{\bm{n}}(\bm{\alpha})\:\frac{1}{c_{\bm{n},\bm{\alpha}}}
\:\prod_{i=1}^d\wt{L}_{n_i,\alpha_i}(x_i)\,
\psi_{n_i,\alpha_i}(s_i)\notag\\
&=\left\{\prod_{i=1}^d(1+s_i)^{-\alpha_i}\right\}\:
\sum_{\bm{n}\in\Z_+^d}\rho_{\bm{n}}(\bm{\alpha}) \prod_{i=1}^d\frac{\wt{L}_{n_i,\alpha_i}(x_i)}{n_i!}
\,\frac{(-s_i)^{n_i}}{(s_i+1)^{n_i}}.
\label{djointl}\end{align}
Note that, replacing each $\rho_{\bm{n}}(\bm{\alpha})$ with 1, the series becomes
\begin{eqnarray}\sum_{\bm{n}\in\Z_+^d}\prod_{i=1}^d\frac{\wt{L}_{n_i,\alpha_i}(x_i)}{n_i!}
\,\frac{(-s_i)^{n_i}}{(s_i+1)^{n_i}}&=&\prod_{i=1}^dG\left(\frac{-s_i}{s_i+1},x_i\right)\notag\\
&=&
\prod_{i=1}^d\left(\frac{1}{1+s_i}\right)^{-\alpha_i}\edr^{-x_is_i},\label{domin}\end{eqnarray}
where $G(r,x)$ is the generating function (\ref{Laggf}). Since
\[
\rho_{\bm{n}}(\bm{\alpha})= \prod_{i=1}^d \frac{1}{c_{n_i,\alpha_i}}\:
\E\left[\prod_{i=1}^d{\wt{L}_{n_i,\alpha_i}(X_i)}{\wt{L}_{n_i,\alpha_i}(Y_i)}\right]
\]
by the Cauchy-Schwarz inequality
\[
|\rho_{\bm{n}}(\bm{\alpha})|^2\leq\prod_{i=1}^d \frac{1}{c_{n_i,\alpha_i}^2}\:\E\left[\wt{L}_{n_i,\alpha_i}^2(X_i)\right]\:
\E\left[\wt{L}^2_{n_i,\alpha_i}(Y_i)\right]=1.
\]
Then $0\leq\rho_{\bm{n}}(\bm{\alpha})\leq 1.$ Hence (\ref{domin}) implies, in particular, that the sum in (\ref{djointl}) converges absolutely.
Now, for every $x>0,$
\begin{multline*}
\E\left[\edr^{-\sum_{i=1}^d s_i\frac{Y_i}{X_i}}\,\mid\, X_i=x, \: i=1,\ldots,d\,\right]=\\[7pt] \prod_{i=1}^d\left(1+\frac{s_i}{x}\right)^{-\alpha_i}\:
\sum_{\bm{n}\in\Z_+^d}\rho_{\bm{n}}(\bm{\alpha})
\prod_{i=1}^d\frac{\wt{L}_{n_i,\alpha_i}(x)}{n_i!}\:
\frac{(-s_i)^{n_i}}{(s_i+x)^{n_i}}=:\phi_{\rho,{x}}(\bm{s}).
 \end{multline*}
 The limit of $\phi_{\rho,{x}}(\bm{s})$, as $x\to\infty$, is still the Laplace transform of a $d$-dimensional probability distribution, again thanks to the boundedness of the canonical correlation coefficients. In particular,
\[
\phi_\rho(\bm{s}):=\lim_{x\to\infty}\phi_{\rho,{x}}(\bm{s})
=\sum_{\bm{n}\in\Z_+^d}\rho_{\bm{n}}(\bm{\alpha})\:\prod_{i=1}^d\frac{(-s_i)^{n_i}}{n_i!}
\]
is continuous at the origin. Thus $\{\rho_{\bm{n}}(\bm{\alpha}):\bm{n}\in\Z_+^d\}$ has the interpretation as the moment sequence of a random variable $\bm{Z}\in\R_+^d.$ Since one further has $\rho_{\bm{n}}(\bm{\alpha})\leq 1$, it is the moment sequence of a random variable in $[0,1]^d.$\\
{Finally, it remains to prove that every canonical correlation sequence corresponds to an algorithm of the type A.2. Let $f$ be any function on $\R_+^d$ with finite $\Gamma_{\bm{\alpha},1}$-variance. Then
\begin{align*}\Ex{f(\bm{Y})\mid \bm{X}=\bm{x}}&=\sum_{\bm{n}\in\Z_+^d}\frac{\hat{f}(\bm{n})}{c_{\bm{n},\bm{\alpha}}}\rho_{\bm{n}}(\bm{\alpha})\wt{L}_{\bm{n},\bm{\alpha}}(\bm{x})\\
&
=\sum_{\bm{n}\in\Z_+^d}
\frac{\hat{f}(\bm{n})}{c_{\bm{n},\bm{\alpha}}}\Ex{\bm{Z}^{\bm{n}}}\wt{L}_{\bm{n},\bm{\alpha}}
(\bm{x})\\
&=\Ex{\sum_{\bm{n}\in\Z_+^d}\frac{\hat{f}(\bm{n})}{c_{\bm{n},\bm{\alpha}}}{\bm{Z}^{\bm{n}}}
\wt{L}_{\bm{n},\bm{\alpha}}(\bm{x})},
\end{align*}
where $\bm{Z}^{\bm{n}}:=\prod_{i=1}^d Z_i^{n_i}.$ The sums can be interchanged because $\sum_{\bm{n}}{c_{\bm{n},\bm{\alpha}}}^{-1}(\hat{f}(\bm{n}))^2<\infty$ and both $|\rho_{\bm{n}}|\leq 1$ and $|\bm{Z}^{\bm{n}}|\leq 1,$ thus all the sums converge. The inner series can therefore be interpreted as a version of the conditional expectation
$$\Ex{f(\bm{Y})\mid \bm{X}=\bm{x},\bm{Z}=\bm{z}}$$ where $(\bm{X},\bm{Y})$ is the pair generated by step (iii) of Algorithm A.2, conditional on the realisation $\bm{B}=\left(z_i/(1-z_i):i=1,\ldots,d\right).$ The proof is thus complete. }
\end{proof}

In Algorithm A.2, given $\bm{X}=\bm{x}$ and $\bm{B}=\bm{b}$, the random variables $Y_1,\ldots,Y_d$ are independent and the distribution of $Y_j$ is a mixture of gamma distributions. Consider the particular case whereby, in Algorithm A.2, $P^*$ is degenerate on $\R_+$ so that with probability 1 one has $\bm{B}=(b,\ldots,b)\in\R_+^d$. Since $d$ independent Poisson random variables, conditioned to their total sum, form a multinomial vector, then an algorithm for simulating realisations of $\bm{Y}$ conditional on $\bm{X}=\bm{x}$ and $\bm{B}=(b,\ldots,b)$ is
\medskip

\textsc{Algorithm A.3.}
\begin{itemize}\addtolength{\itemsep}{2.5pt}
\item[\texttt{(0)}] Initialise by fixing $b\in\R_+,\bm{x}\in\R_+^d,\ \bm{\alpha}\in\R_+^d$.
\item[\texttt{(i)}] Sample $N=n$ from a Poisson distribution with parameter $b\,|\bm{x}|$.
\item[\texttt{(ii)}] Sample $N_1=n_1,\ldots,N_d=n_d$ from a multinomial distribution with parameter vector $(n;\,x_1/|\bm{x}|,\,\ldots,x_d/|\bm{x}|)$
\item[\texttt{(iii)}] Sample $\bm{Y}$ from $\Gamma_{\bm{\alpha}+\bm{n},\,(1+b)^{-1}}=\times_{i=1}^d\Gamma_{\alpha_i+n_i,(1+b)^{-1}}$
\end{itemize}
\medskip

\begin{exe}\label{sec:dCSBI}
{\rm  ($d$-dimensional CSBI). 
 Consider a $\R_+^d$-valued process $(\bm{X}_t:t\geq 0)=((X_{1,t},,\ldots,X_{d,t}):t\geq 0)$ whose coordinates $(X_{i,t}:t\geq 0)$ $(i=1,\ldots,d)$ are independent continuous-state branching processes with immigration with stationary gamma$(\alpha_i,1)$ distribution, respectively. The generator of each coordinate is of the form \eqref{1csbi}. Then the generator of the vector is the sum of the marginal generators and the transition function of $(\bm{X}_t)$ is simply derived as the product of the individual coordinates' transition functions:
    \begin{eqnarray}
      p^{\alpha_1,\ldots,\alpha_d}_t(\bm x, \ddr\bm y)&=&\prod_{i=1}^d p^{\alpha_i}_t(x_i,\ddr y_i)\notag\\
                                                      &=&\sum_{n_1}\cdots\sum_{n_d} e^{- n_i /2}\wt{L}_{n_i,\alpha_i}(x_i)\wt{L}_{n_i,\alpha_i}(y_i)\Gamma_{\alpha_i,1}(\ddr y_i)\notag\\
                                                      &=&\sum_{\bm{n}}e^{-|\bm{n}| t/2}\wt{L}_{\bm{n}.\bm{\alpha}}(\bm{x})\wt{L}_{\bm{n}.\bm{\alpha}}(\bm{y})\Gamma_{\bm{\alpha},1}(\ddr \bm{y}).\label{dCSBI_tf}
    \end{eqnarray}
    Then, for each $t$, the process has canonical correlations of the form $\rho_{n_1,\ldots,n_d}(t)=e^{-|\bm{n}| t/2}$. Therefore it is easy to see that an algorithm for sampling from the transition function \eqref{dCSBI_tf} is given by Algorithm A.3 for $b=z(1-z)^{-1}$ and $z=e^{-t/2}$.\\
    This class has an important infinite-dimensional extension, the
    $\Gamma$-Dawson-Watanabe processes mentioned in the
    Introduction. It will be studied later in Sections 5.1 and play an
    important role in Section \ref{sec:markovcc}.}
  \end{exe}

\section{Dependent gamma CRMs}
Let us now move on to characterizing dependent vectors of gamma CRMs that are in canonical correlation. Since we are now dealing with infinite--dimensional objects, we need to make precise what we mean by a vector of random measures in canonical correlation. To simplify notation, for any collection $\Ac=\{A_1,\ldots,A_d\}$ of pariwise disjoint measurable subsets of $\Xc$ we set $$\bm{\wt{\mu}}_{j,\Ac}:=(\wt{\mu}_j(A_1), \ldots, \wt{\mu}_j(A_d))$$for each $j=1,2$. Denote by $\Xcr^*$ the set whose elements are all the finite collections of disjoint Borel sets of $X,$
and set $\Z_+^*=\bigcup_{d\geq 1}\Z_+^d.$ Throughout this and the next Sections the symbol $\Ac$ will be used for a generic element of $\Xcr^*$ and by $d(\Ac)$ the cardinality of $\Ac$.
\begin{defi}\label{def:ccrm}
A pair $(\wt\mu_1,\wt\mu_2)$ of random measures, with $\wt\mu_1\stackrel{d}{=}\wt\mu_2$, is in canonical correlation if $\bm{\wt{\mu}}_{1,\Ac}$ and $\bm{\wt{\mu}}_{2,\Ac}$ are canonically correlated, for every  $\Ac\in \Xcr^*$. The corresponding canonical correlation sequences in \eqref{ccdef1} shall be denoted by $\rho(\Ac):=\{\rho_{\bm{n}}(\Ac):\: \bm{n}\in\Z_+^{d(\Ac)}\}$. We denote the collection of all such sequences $\{\rho_{\bm{n}}(\Ac):\: \bm{n}\in\Z_+^{d(\Ac)},\Ac\in\Xcr^*\}$ by $\rho$.
\end{defi}

Notice that the finite-dimensional distributions of $(\wt\mu_1,\wt\mu_2)$, for non-disjoint sets, can always be expressed in terms of (polynomials of) finite-dimensional distributions on disjoint sets, therefore, by linear extension, Definition \ref{def:ccrm} fully describes the dependence in $(\wt\mu_1,\wt\mu_2)$. For pairs of Gamma$(c,P_0)$ CRMs, the coefficients $\rho_{\bm{n}}(\Ac)$ will be expectations of products of polynomials $\wt{L}_{n,cP_0(A_i)}$ $(i=1,\ldots,d(\Ac))$. Furthermore, notice that Definition \ref{def:ccrm} involves any choice of disjoint Borel sets $\mathcal{A}=\{A_1,\ldots,A_d\}$ but it is sufficient to check it only for those collections such that $P_0(A_i)>0,\ i=1,\ldots,d.$ Indeed, if one of the sets, $A_i$ say, has $P_0$ null measure, then $\wt{L}_{n,cP_0(A_i)}=\delta_{n,0}$ and therefore
$$
\rho_{n_1,\ldots,n_d}(A_1,\ldots,A_d)=\rho_{n_1,\ldots,n_{i-1},n_{i+1},\ldots,n_d}(A_1,\ldots,A_{i-1},A_{i+1},\ldots,A_d).
$$


Of course, if $(\wt \mu_1,\wt \mu_2)$ is a pair of gamma measures in canonical correlation, then Theorem \ref{thm:dprodcc} indicates that, for any $d$ and $\Ac=(A_1,,\ldots,A_d)$ the canonical correlation coefficients $\rho(\Ac)$ of $(\wt\mu_1(\Ac),\wt\mu_2(\Ac))$ must be the moment sequence of a random vector $(Z_{A_1},\ldots,Z_{A_d})\in[0,1]^d.$ The first result shows an important identity in distribution among all such $Z_A,$ as $A$ varies in $\Xcr,$ induced by the Kolmogorov consistency property of $(\wt \mu_1,\wt \mu_2)$. This will pave the way for the main Theorem \ref{thm:main} of this Section.

\begin{lem}
\label{l:zadditive}
Let $(\wt\mu_1,\wt\mu_2)$ be a vector of gamma CRMs with parameter measure $cP_0$. $(\wt\mu_1,\wt\mu_2)$ is in canonical correlation if and only if there exists a sequence of $[0,1]$-valued random variables $\{Z_A:A\in\Xcr\}$ such that for every $d\ge 1$ and collection $\Ac=\{A_1,\ldots,A_d\}$ of pairwise disjoint sets in $\Xcr,$ the finite-dimensional vectors $\bm{\wt{\mu}}_{1,\Ac}$ and $\bm{\wt{\mu}}_{2,\Ac}$ are canonically correlated via the sequence of moments
\begin{equation}
\rho_{\bm{n}}(\Ac)=\E\left[\prod_{i=1}^d Z_{A_i}^{n_i}\right]
\label{eq:moment}
\end{equation}
Moreover, for any disjoint pair $A,B\in\Xcr$,
\begin{equation}
Z_{A\cup B}\overset{d}{=}\ep_{A,B}Z_A+(1-\ep_{A,B})Z_B
\label{eq:zi}
\end{equation}
where $\ep_{A,B}$ is a $\mbox{beta}(cP_0(A),cP_0(B))$ random variable, independent from $\{Z_A,Z_B\}$.
\end{lem}

\begin{proof}
Suppose $(\wt\mu_1,\wt\mu_2)$ is in canonical correlation  and introduce simple functions $f=\sum_{i=1}^d s_i\indic_{A_i}$ and $g=\sum_{i=1}^d t_i\indic_{A_i}$, where $A_1,\ldots,A_d$ are measurable disjoint subsets of $\Xc$. From Theorem~\ref{thm:dprodcc}, 
 one then has
\begin{align}
  \E\left[\edr^{-\wt\mu_1(f)-\wt\mu_2(g)}\right]
  &=
  \sum_{\bm{n}\in\Z_+^d}\rho_{\bm{n}}(\Ac)\:\prod_{i=1}^d\frac{\psi_{n_i,\alpha_i}(s_i)\:
  \psi_{n_i,\alpha_i}(t_i)}{(\alpha_i)_{n_i}\: n_i!}\notag\\[7pt]
  &=\left\{\prod_{i=1}^d \phi_{\alpha_i,1}(s_i)\,\phi_{\alpha_i,1}(t_i)\right\}\:
  \sum_{\bm{n}\in\Z_+^d}\rho_{\bm{n}}(\Ac)\:\prod_{i=1}^d \frac{(\alpha_i)_{n_i}}{n_i!}
  (\theta_i)^{n_i}\notag\\[7pt]
  &=\left\{\prod_{i=1}^d \phi_{\alpha_i,1}(s_i)\,\phi_{\alpha_i,1}(t_i)\right\}\:
  \E\left[\prod_{i=1}^d\sum_{n_i=0}^\infty \frac{(Y_iZ_{i})^{n_i}}{n_i!}\:\theta_i^{n_i}\right]
\label{eq:dexpansion}\end{align}
for some random vector $(Z_1,\ldots,Z_d)\in[0,1]^d,$ where $\theta_i=s_it_i/[(1+s_i)(1+t_i)]$, for each $i=1,\ldots,d$, and the $Y_i$s are independent random variables with respective distributions $\Gamma_{\alpha_i,1}$, also independent of the $Z_i$s. The series expansion in the expected value above leads to
\begin{equation}
  \label{eq:rand_meas}
  \E\left[\edr^{-\wt\mu_1(f)-\wt\mu_2(g)}\right]=\left\{\prod_{i=1}^d \phi_{\alpha_i,1}(s_i)\,\phi_{\alpha_i,1}(t_i)\right\}\:\int_{(\R_+)^d}\edr^{\langle \bm{\theta},\bm{v}\rangle}\: p_{\Ac}(\ddr\bm{v})
\end{equation}
where $p_{\Ac}$ is a probability distribution on $(\R_+)^d$ such that
\begin{equation}
\int_{(\R_+)^d}\: \bm{v}^{\bm{n}}\:p_{\Ac}(\ddr\bm{v})=
\E\left[\prod_{i=1}^d Z_i^{n_i}\right]\:
\prod_{i=1}^d (\alpha_i)_{n_i}
\label{eq:kimom_1}
\end{equation}
and $\bm{v}^{\bm{n}}=\prod_{i=1}^d v_i^{n_i}$. Now choose any two indices $i,j\in[d]=\{1,\ldots,d\}$ and set $s_i=s_j$ and $t_i=t_j$. The very same argument leading to (\ref{eq:rand_meas}) allows us to identify new random variables $Z_{ij}$ and $(Z^\prime_k:k\neq i,j)$, say, via
\begin{multline}
\label{eq:kimom_2}
\int_{(\R_+)^{d-1}}\: \bm{v}^{\bm{n}}\:p_{\wt \Ac}(\ddr\bm{v})
=\E\left[\left(Z_{ij}^{n_i+n_j}\right)\prod_{k\neq i,j}{Z^\prime}_k^{n_k}\right]\\
\times\:(\alpha_i+\alpha_j)_{n_i+n_j}\left[\prod_{k\neq i,j}(\alpha_k)_{n_k}\right].
\end{multline}
Note that (\ref{eq:kimom_1}) and (\ref{eq:kimom_2}) need to be identical (by construction and by (\ref{eq:rand_meas}) with $s_i=s_j,t_i=t_j$). In particular, on one hand the vector $(Z^\prime_k:k\in[d],k\neq i,j)$ must be equal in distribution to the vector $(Z_k:k\in[d], k\neq i,j)$ (setting $n_i=n_j=0$ in \eqref{eq:kimom_1} and \eqref{eq:kimom_2} shows that all their moments coincide). On the other hand, setting $K_i=Z_iY_i$ and $K_{ij}=Z_{ij}(Y_i+Y_j)$ we thus see that 
$$\E[K_{ij}^{n}]=\E[(K_i+K_j)^n]$$
But since $Y_i$ has distribution $\Gamma_{\alpha_i,1}$ and is independent of $(Z_1,\ldots,Z_d)$ for every $i$, this also implies
that
\begin{eqnarray}
\E[Z_{ij}^{n}]&=&\frac{\E[K_{ij}^{n}]}{(\alpha_i+\alpha_j)_{(n)}}=\frac{\E[(K_i+K_j)^n]}{(\alpha_i+\alpha_j)_{(n)}}\notag\\
&=&\sum_{k=0}^n{n\choose j}\frac{(\alpha_i)_{(j)}(\alpha_j)_{(n-k)}}{(\alpha_i+\alpha_j)_{(n)}}\E[Z_i^{n_i}Z_j^{n_j}]\notag\\
&=&\E[(\epsilon_{i,j}Z_i+(1-\epsilon_{i,j})Z_j)^n]\notag
\end{eqnarray}
where $\epsilon_{i,j}\sim \mbox{Beta}(\alpha_i,\alpha_j),$ is independent of $(Z_i,Z_j)$. Equivalently,
$$Z_{ij}\overset d = \epsilon_{i,j}Z_i+(1-\epsilon_{i,j})Z_j.$$ Now we can identify, in the previous construction, $Z_{A_i}:=Z_i, i=1,\ldots,d,$ and $Z_{A_i\cup A_j}:=Z_{ij}$ and see that the above  distributional equation coincides with (\ref {eq:zi}) for $A=A_i,B=A_j$. But since the choice of $\Ac, i,j$ was arbitrary, we have proved that (\ref {eq:zi}) holds for all $A,B\in\Xcr$ and, by induction, we have proved the necessity part of the Lemma.\\
To prove sufficiency, assume that there exists a set of $[0,1]$-valued random variables $\{Z_A:A\in\Xcr\}$ satisfying (\ref {eq:zi}). For every $d$ and disjoint sets $\Ac=\{A_1,,\ldots, A_d\},$ the mixed moments $\Ex{Z_{A_1}^{n_1}\cdots Z_{A_d}^{n_d}}$ identify a sequence of canonical correlation coefficients for a pair $(\bm{X}(\Ac),\bm Y(\Ac))$ of random vectors with identical marginal distribution $\times_{i=1}^d \Gamma_{\alpha_i,1}$, where $\alpha_i=c P_0(A_i), i=1,\ldots,d.$ The Laplace transform of $(\bm{X}(\Ac),\bm Y(\Ac))$ is of the form \eqref{eq:dexpansion}, for any pair of functions $f=\sum_1^d s_i X_i(\Ac)$ and $g=\sum_1^d t_i Y_i(\Ac)$, where $Z_i=Z_{A_i}$, $\bm{X}(\Ac)=(X_1(\Ac),\ldots,X_d(\Ac))$ and $\bm{Y}(\Ac)=(Y_1(\Ac),\ldots,Y_d(\Ac))$.
Denote by $\phi_{\Ac}(\bm{s},\bm{t})$ such a Laplace transform. We now show that the property  (\ref {eq:zi}) makes it possible for $\phi_{\Ac}$ to identify, as $\Ac$ varies, the finite-dimensional distributions of a pair of random measures with identical $\Gamma_{cP_0}$-marginals. First of all, if $\Ac_-=\Ac\setminus A_d$ one easily notices that
\begin{equation}
\phi_{\Ac}(s_1,\ldots,s_{d-1},0;t_1,\ldots,t_{d-1},0)
=\phi_{\Ac_-}(s_1,\ldots,s_{d-1};t_1,\ldots,t_{d-1}).
\label{eq:consist}
\end{equation}
Furthermore, 
if $A_i\downarrow \emptyset$, for some $i=1,\ldots, d$, then $P_0(A_i)\to 0$ 
and
\begin{equation}
\lim_{A_i\downarrow \emptyset}
\phi_{\Ac}(\bm{s};\bm{t})
=\phi_{\Ac_{-i}}(\bm{s}_{-i,};\bm{t}_{-i})
\label{eq:emptylim}
\end{equation}
where $\Ac_{-i}=\Ac\setminus A_{i}$, while $\bm{s}_{-i,}$ and $\bm{t}_{-i}$ and the $\bm{s}$ and $\bm{t}$ vectors with the $i$--th component removed. In particular, for $d=1,$
$\phi_{1,\alpha_1}(s;t)\to 1$ as $A_1\downarrow\emptyset$. Finally, when $d=2$, note that
\begin{equation}
  \begin{split}
    \phi_{\Ac}(s,s;t,t)
    &=
    \prod_{i=1}^2
    \phi_{\alpha_i,1}(s)\,\phi_{\alpha_i,1}(t)\:
    \E\left[\prod_{i=1}^2\sum_{n_i=0}^\infty \frac{(Y_iZ_{A_i})^{n_i}}{n_i!}\:\theta^{n_i}\right]\\[4pt]
    &=
    \phi_{\alpha_1+\alpha_2,1}(s)\,\phi_{\alpha_1+\alpha_2,1}(t) \\[4pt]
    &\quad\times\:
    \E\left[\sum_{n=0}^\infty \frac{(Y_1+Y_2)^n}{n!} \right.\\[4pt]
    &\qquad\times\:\left.\theta^n \sum_{k=0}^n {n\choose k} (XZ_{A_1})^k[(1-X)Z_{A_2}]^{n-k}\right], \label{add1lep}
  \end{split}
\end{equation}
 where $X:=Y_1/(Y_1+Y_2)$ is a Beta$(\alpha_1,\alpha_2)$ random variable, independent of $(Z_{A_1},Z_{A_2})$ and of $(Y_1+Y_2)$, the latter having distribution $\Gamma_{\alpha_1+\alpha_2,1}$. By \eqref{eq:zi}, the expectation on the right hand side of \eqref{add1lep} can then be written as
$$\E\left[\sum_{n=0}^\infty \frac{(Y_{12}Z_{A_1\cup A_2})^n}{n!} \theta^n  \right]$$
with $Y_{12}\overset d = Y_1+Y_2$ and one may then conclude that
\begin{equation}
  \label{eq:union}
      \phi_{\Ac}(s,s;t,t)=\phi_{\{A_1\cup A_2\}}(s,t).
\end{equation}
The properties \eqref{eq:consist},\eqref{eq:emptylim} and \eqref{eq:union} effectively coincide with the necessary and sufficient conditions for $\phi_{\Ac}$ to determine the full set of finite-dimensional distributions of a uniquely defined random measure on $\Xc\times\Xc$ (see conditions 9.2.VI  Lemma 9.2.IX of \cite{daley}). Since all such finite-dimensional distributions correspond to gamma vectors in canonical correlations, this proves the sufficiency part of the claim. \end{proof}

\subsection{Random Dirichlet means.}
We first recall that a Dirichlet process with base measure $c P_0$, where $c>0$ and $P_0$ is some probability measure on $(\Xc\,Xcr)$, can be defined as the normalisation of a Gamma process $\wt\mu$ on $(\Xc,\Xcr)$ with parameters $(c,P_0)$, by its total mass $\wt\mu(\Xc)$. It will henceforth be denoted as $\wt p_{cP_0}$. A nice and surprising consequence of Lemma~\ref{l:zadditive}, is a connection between canonical correlations in \eqref{eq:moment} and linear functionals of a Dirichlet process.
If for any $d\ge 1$ the collection $\{A_1,\ldots,A_d\}$ is a measurable partition of $A\in\Xcr$, then \eqref{eq:zi} can be trivially extended to
\begin{equation}
Z_A\stackrel{d}{=}\sum_{i=1}^d \ep_i\:Z_{A_i}\qquad \mbox{with}\qquad \ep_i=\frac{\wt\mu(A_i)}{\wt\mu(A)}
\label{eq:epsilon}
\end{equation}
which entails that $Z_A\eqd \int f_d(x)\, \wt p_{H^*_A}(\ddr x)$, where $f_d(x)=\sum_{i=1}^d Z_{A_i}\,\mathds{1}_{A_i}(x)$ and $H_A^*(\,\cdot\,)= c P_0(\,\cdot\,\cap A)$. Since such a representation holds for every $d$, it suggests that, at least in the case in which all $Z_{A_i}$ are independent, taking the limit as $d\to\infty,$ $Z_A$ can be seen as a linear functional of a Dirichlet process, whose parameter measure is supported by $A$.  This is summarised by the following result.

\begin{thm}\label{thm:zeta}
Let $(\wt\mu_1,\wt\mu_2)$ be a pair of Gamma random measures, with parameter $(c,P_0)$, in canonical correlation 
Then 
$(\wt\mu_1,\wt\mu_2)$ has independent increments if and only if there exists a probability kernel $Q$ on $\Xc\times\mathcal{B}([0,1])$ such that, for every $A\in\Xcr$, the canonical correlation coefficients are moments $\rho_n(A)=\Ex{Z_A^n}$ where,
\begin{equation}
  \label{eq:Z_A}
  Z_A\overset d= \int_0^1 s \ \wt{p}_{A}(\ddr s),\ \ \ A\in\Xcr
   \end{equation}
with $\wt{p}_{A}$ a Dirichlet process with parameters $(cP_0(A),G_A)$  and
$$G_A(\ddr s):=\frac{\int_A  \ Q(x; \ddr s)P_0(dx)}{P_0(A)}. $$


\end{thm}

\begin{proof}
One direction of the result is immediate. If the $Z_A$'s are independent, for any collection $\Ac$ of pairwise disjoint sets, and defined as in \eqref{eq:Z_A}, they satisfy \eqref{eq:epsilon} and identify the canonical correlations $\rho_{\bm{n}}(\Ac)$ of gamma CRM vector whose increments are independent.

The proof of necessity is trickier. In the following we denote $H:=cP_0$ and note that, if $(\wt\mu_1,\wt\mu_2)$ has independent increments, then the canonical correlation coefficients must be of the form $\rho_{n_1,n_2}(A,B)=\rho_{n_1}(A)\rho_{n_2}(B),$ for every pair $A$ and $B$ of disjoint sets in $\Xcr$. This, combined with \eqref{eq:moment},  implies that $Z_A$ and $Z_B$ are independent. It can be further shown that for every sequence $\{A_n:\: n\ge 1\}$ of pairwise disjoint sets in $\Xcr$, the random variables $\{Z_{A_i}:\: i\ge 1\}$ are mutually independent. Introduce, now, a sequence  $\Pi:=(\Pi_m)_{m\ge 1}$ of nested and measurable partitions of $A$ in such a way that $\Pi_m:=\{A_{\bm{\ep}}:\: \bm{\ep}\in\{0,1\}^m\}$ and, for any $\bm{\ep}\in\{0,1\}^m$, one has
$A_{\bm{\ep}}=A_{\bm{\ep}0}\cup A_{\bm{\ep}1}$. From \eqref{eq:epsilon} 
one trivially deduces
\begin{equation}
  \label{eq:decomp}
  Z_A\stackrel{d}{=}\sum_{\bm{\ep}\in\{0,1\}^m} Z_{A_{\bm{\ep}}}\,
  \tilde \pi_A(A_{\bm{\ep}})
\end{equation}
where $\tilde \pi_A$ is a Dirichlet process with parameters $(cP_0(A),H_A)$ and $H_A(\cdot):=H(\,\cdot\,\cap A)/H(A)$. Moreover, all $Z_{A_{\bm{\ep}}}$ in the sum are independent. The stick--breaking representation of the Dirichlet process, as established in \cite{S94}, allows us to set 
\[
\pi_A(A_{\bm{\ep}})=\sum_{j=1}^\infty \wt{P}_{A,j}\delta_{Y_j}(A_{\bm{\ep}})
\] 
where $(\wt{P}_{A,j})_{j\ge 1}$ is a so--called GEM$(H(A))$ sequence meaning that $(V_J)_{j\ge 1}$, with $V_j=\wt P_{A,j}/[1-\sum_{i=1}^{j-1}\wt P_{A,i}]$, is a sequence of iid random variables whose common distribution is Beta$(1,H(A))$.  Moreover, $(Y_j)_{j\ge 1}$ is a sequence of iid $\Xc$--valued random variables, independent of $(\wt{P}_{A,j})_{j\ge 1},$ with common distribution $H_A$. In view of this, one can rewrite (\ref{eq:decomp}) as follows
\begin{equation}
\label{eq:zseries}
Z_A\stackrel{d}{=}\sum_{j=1}^\infty \wt{P}_{A,j}\,U_{m,j},\ \ m=1,2,\ldots
\end{equation}
where 
$U_{m,j}:=\sum_{\bm{\ep}\in\{0,1\}^m} Z_{A_{\bm{\ep}}}\delta_{Y_j}(A_{\bm{\ep}})$  for any $m$ and $j$ in $\N$. Note that, for every $m$, the sequences $(U_{m,1},U_{m,2},\ldots)$ and $(U_{m,2},U_{m,3},\ldots)$ have the same distribution. Then we can rewrite \eqref{eq:zseries} as
\begin{equation}
Z_A\eqd \wt P_{A,1}U_{m,1}+(1-\wt P_{A,1})Z^{*}_{m,A},\ \ m=1,2,\ldots
\end{equation}
where, for every $m$, $\wt P_{A,1}$ is independent of $(U_{m,1},Z^{*}_{m,A})$ and $Z^{*}_{m,A}\overset d = Z_A.$ 
Moreover, since for any $m$
\begin{align*}
\mathds{E} Z_A^n&=
\sum_{j=0}^n{n\choose j}\Ex{P_{A,1}^j(1-\wt P_{A,1})^{n-j}}\Ex{U_{m,1}^j{Z^{*}_{m,A}}^{n-j}}\\
&=\sum_{j=0}^n{n\choose j}\frac{j!(\theta)_{(n-j)}}{(1+\theta)_{(n)}}
\:\mathds{E} U_{m,1}^j{Z^{*}_{m,A}} ^{n-j}
\end{align*}
and the moments of $Z^{*}_{m,A}$ do not depend on $m$, one has that $U_{m,1}$ converges in distribution to a random variable in $[0,1]$, as $m\to\infty$. 
In view of , to conclude the proofSet
 $$E_{m,j}:=\sum_{\bm{\ep}\in\{0,1\}^{m}}\bm{\epsilon}\: \indic_{A_{\bm{\ep}}}(Y_j),\qquad
 j\geq 1.$$
For each $m,j\in\N,$ $E_{m,j}$ identifies the unique set in the partition $\Pi_m$ that contains $Y_j$, for every $j\geq 1$, namely $U_{m,j}=Z_{A_{E_{m,j}}}.$ 
Note that, if $P_0$ is diffuse, then
$$\Prob{E_{m,i}=E_{m,j}}=\Ex{\sum_{\ep\in\{0,1\}^m}\max\{\indic_{A_{\ep}}(Y_i), \indic_{A_{\ep}}(Y_j)\}}\to 0\ \ \ m\to\infty.$$
Thus, for every $C$ and $D$ in ${\cal B}([0,1])$, 
\begin{align*}
\lim_{m\to\infty} \Prob{U_{m,1}\in C, U_{m,j}\in D}&= \lim_{m\to\infty}\Prob{Z_{A_{E_{m,1}}}\in C, Z_{A_{E_{m,j}}}\in D\mid\ E_{m,1}\neq E_{m,j}}\\[4pt]
&=\lim_{m\to\infty}\Prob{Z_{A_{E_{m,1}}}\in C}\Prob{Z_{A_{E_{m,j}}}\in D}
\end{align*}
because  the $(Z_A, Z_B)$ are independent for any disjoint $A,B$ by assumption. The same can be seen by taking triplets, quadruplets, etc., of $JU_{m,i}$'s and the limiting sequence $(U_{j})$ is, thus, iid.  so that we can write, for every $A\in\Xcr$, 
\begin{equation}
Z_A  \overset{d}=\sum_{j=1}^\infty \wt{P}_{A,j}\,U_{j}
\eqd \int_0^1 x\ \wt{p}_{A}(\ddr x)\label{eq:mean_z}
\end{equation}
where $(\tilde{P}_{A,j})_{j\ge 1}$ and $(U_{j})_{j\ge 1}$ are independent and $\wt{p}_{A}$ is a Dirichlet process on $[0,1]$ with parameter $(H(A), G_{A}),$ for some probability measure $G_A(\cdot)$, being the distribution of $U_1$.
The second distributional equation follows from the series representation of the Dirichlet process. 
If $B$ in $\Xcr$ is such that $A\cap B=\varnothing$, one then has tat $\tilde p_A$ and $\tilde p_B$ denote independent Dirichlet processes and by 
virtue of \eqref{eq:zi} 
$$Z_{A\cup B}\overset d = \int_0^1 x \ \wt{p}_{{A\cup B}}(\ddr x)\overset d=\epsilon_{A,B}\int_0^1 x \ \wt{p}_{A}(\ddr x)+(1-\epsilon_{A,B})\int_0^1 x \ \wt{p}_{B}(\ddr x)$$
where $\wt p_{A\cup B}$ is a Dirichlet process on $[0,1]$ with parameters $(H(A\cup B), G_{A\cup B})$, for some probability distribution $G_{A\cup B}$ restricted to $A\cup B$,  and $\epsilon_{A,B}$ is beta--distributed with parameters $(H(A),H(B))$ and is independent from both $\wt{p}_{A}$ and $\wt{p}_{B}$. 
Therefore it has to be
$$ H(A\cup B) G_{A\cup B}(\cdot)=H(A) G_{A}(\cdot)+H(B) G_{B}(\cdot)$$
for every $A,B\in\Xcr$ disjoint. In other words, there is a measure $G$ on the product space $\Xcr\times{\cal B}([0,1])$ satisfying
$G(A\times [0,1])=H(A)$  and $G_A(C)=G(A\times C)/H(A)$ for every $A\in\Xcr$ and $C\in{\cal B}([0,1])$. Correspondingly, there will be a probability kernel $Q(x,\ddr z)$ on $\Xcr\times{\cal B}([0,1])$ such that
$$G(A\times C)=\int_{A\times C}H(\ddr x)Q(x,\ddr s)$$
which is what we wanted to prove.
\end{proof}

\subsection{Laplace functional.}

We are now in a position to state then main result of this section: gamma $(c,P_0)$ random measures in canonical correlation form a convex set whose extreme points are indexed by certain linear functionals of a gamma $(c,P_0)$ completely random measure.

\begin{thm}\label{thm:main}
Let $(\wt\mu_1,\wt\mu_2)$ be a vector of gamma CRMs with parameter measure $cP_0$. 
$(\wt\mu_1,\wt\mu_2)$ is in canonical correlation if and only if its joint Laplace functional 
is as in \eqref{eq:repres_biv_lapl},
for a random measure $K$  on $(\Xc,\Xcr)$ with exchangeable increments determined as follows: there exists a (possibly random) probability kernel $Q(x,\ddr z)=Q_x(\ddr z)$ on $\Xcr\times{\cal B}([0,1])$ such that,
\begin{equation}
\label{eq:kmeasurenew}
K(f)\overset{d}=\int_{\Xc\times[0,1]} s f(x)\ \wt\mu(\ddr x,\ddr s)
\end{equation}
where $\wt\mu$ is, conditionally on $Q$, a Gamma random measure on $\Xc\times [0,1]$ with parameter $(c,G)$ where $G(\ddr x,\ddr s):=cP_0(\ddr x)Q_x(\ddr s).$ Then $(\wt\mu_1,\wt\mu_2)$ defines a CRM vector if and only if the underlying kernel $Q$ is deterministic.
\end{thm}

\begin{proof}

We want to prove that 
\eqref {eq:zi}
 implies the \lq\lq only if\rq\rq  part of the claim. To this end, consider the measures $p_{\Ac}$ defined by \eqref{eq:kimom_1} in the proof of Lemma \ref{l:zadditive}. We show that the collection $\Pcr':=\{p_{\Ac}:\: \Ac\in\Xcr^d,\: d\ge 1\}$ uniquely identifies a random measure on $(\Xc,\Xcr)$. Notice that the system $\Pcr'$ can be extended to define a family $\Pcr=\{p_{B_1,\ldots,B_d}:\: B_i\in\Xcr,\: d\ge 1\}$ in such a way that $\Pcr$ is consistent and
\[
p_{\Bcr}(\{(x,y,z)\in(\R_+)^3: \: x+y=z\})=1
\]
for any $A$ and $B$ in $\Xcr$ such that $A\cap B=\varnothing$ and $\Bcr=\{A,B,A\cup B\}$. Moreover, if $d=1$ and $\Ac=\{A_1\}$, $p_{\Ac}$ coincides with the probability distribution of $Z_{A_1}\,\wt\mu(A_1)$, where $\wt\mu$ is a $Gamma$ CRM with parameters $(c,P_0)$, independent of $Z_{A_1}$. Hence, if $(B_n)_{n\ge 1}$ is a sequence of elements in $\Xcr$ such that $\lim_n B_n=\varnothing$ and $\Bcr_n=\{B_n\}$ for any $n\ge 1$, $p_{\Bcr_n}\Rightarrow \delta_0$ almost surely, where $\Rightarrow$ stands for weak convergence. By Theorem~5.4 in \cite{kallenberg} there exists a unique random measure, say $K$, on $(\Xc,\Xcr)$ admitting $\Pcr$ as its system of finite--dimensional distributions. \\
Such a measure $K$ has necessarily exchangeable increments. Indeed, chose $\Ac=(A_1,\ldots,A_d)$ disjoint such that $cP_0(A_i)=cP_0(A_j)$ for $i,j=1,\ldots,d$. Then

\[
\E\left[\prod_{i=1}^d{\wt{L}_{n_i,\alpha_i}(X_i)}{\wt{L}_{n_i,\alpha_i}(Y_i)}\right]=
\E\left[\prod_{i=1}^d{\wt{L}_{n_i,\alpha_{\sigma(i)}}(X_i)}{\wt{L}_{n_i,\alpha_{\sigma(i)}}(Y_i)}\right]
\]
 for any permutation $\sigma$ of $(1,\ldots,d)$. Hence
 
 \[
 \rho_{\bm{n}}(\Ac)= \rho_{\bm{n}}(\sigma \Ac)\ \ \forall\sigma
 \]
(in particular, by \eqref{eq:moment}, the vector $(Z_{A_1},\ldots,Z_{A_d})$ is exchangeable). Thus the vector $(K(A_1),\ldots,K(A_d))$ (whose Laplace transform is given by the expectation of \eqref{eq:dexpansion}), is exchangeable, hence $K$ has exchangeable increments. This means that $K$ can be seen as a convex linear combination of measures with independent increments (see e.g. \cite{KA05}, Proposition 1.21). Theorem \ref{thm:zeta} provides a characterisation for the extreme points of such convex combination: in the case of $K$ with independent increments, the coordinates of
$$(K(A_1),\ldots,K(A_d))\overset d  = (Z_{A_1}\wt\mu(A_1),\ldots, Z_{A_2}\wt\mu(A_d))$$
are independent and 
$$Z_A=\int_{[0,1]}z\ \wt p_{A}(dx)$$
are Dirichlet Process means with the same notation as in Theorem \ref{thm:zeta}. Using $\theta_d=\sum_{i=1}^d \theta_i\indic_{A_i}$ and the $d$ sets $A_1,\ldots,A_d$ in $\Xcr$ being pairwise disjoint, \eqref{eq:dexpansion}) 
induces the following expansion for the Laplace functional of $K$: letting $Y_i=\wt\mu(A_i)$ $(i=1,\ldots,d)$
\begin{eqnarray}
\Ex{e^{K(\theta_d)}}&=&\Ex{\prod_{i=1}^d\sum_{n_i=0}^d\frac{(\theta_i Y_i Z_{A_i})^{n_i}}{n_i!}}\notag\\
&=&\Ex{\prod_{i=1}^d e^{\theta_i Y_i Z_{A_i}}}=\prod_{i=1}^d\Ex{(1-\theta_i Z_{A_i})^{-cP_0(A_i)}}.\label{eq:klap_1}
\end{eqnarray}
Using the so--called Markov-Krein, or Cifarelli--Regazzini, identity (see \cite{lancillotto} and \cite{LPsurvmeans}) \eqref{eq:klap_1} can be rewritten as
\begin{equation}
\Ex{e^{K(\theta_d)}}=\exp\left\{-\sum_{i=1}^d\int_{[0,1]}\ \log(1-\theta_i s)\ cP_0(A_i)G_{A_i}(\ddr s)\right\}\end{equation}
So, if $\theta_d\to\theta$ pointwise as ${d\to \infty}$, then 
$$\Ex{e^{K(\theta_d)}}\to\exp\left\{-c\int_{\X\times[0,1]} \log(1-\theta(x)s)\ P_0(\ddr x)\ Q_x(\ddr s) \right\}$$
for a kernel $Q_x$ as given in Theorem \ref{thm:zeta}. The form of this exponent indicates that
$$K(\theta)=\int_{\X\times[0,1]} s \ \theta(x)\ \wt\mu(\ddr x,\ddr s),$$
with $\wt\mu$ being a Gamma CRM with parameter measure $cP_0(dx)Q_x(dz).$
In other words, 
$$K(\ddr x)=\int_0^1 s \ \wt\mu(\ddr x,\ddr s).$$
The converse of the theorem is much simpler.
Assume that the Laplace functional of $(\mu_1,\mu_2)$ is of the form \eqref{eq:repres_biv_lapl} with a measure $K$ satisfying \eqref{eq:kmeasurenew} with a random kernel Q. With $f=\sum_{i=1}^d s_i\indic_{A_i}$ and $g=\sum_{i=1}^d t_i\indic_{A_i}$ and the $d$ sets $A_1,\ldots,A_d$ in $\Xcr$ being pairwise disjoint, the Laplace transform $\phi(f,g)$  in \eqref{eq:repres_biv_lapl} is of the form (\ref{eq:rand_meas}), and since in \eqref{eq:klap_1} $Z_A$ and $\wt\mu(A)$ are still independent even if, when $Q$ is random, the collection $(Z_A)$ does not necessarily have independent coordinates. For every $A\in\Xcr$, an expansion for $\phi(f,g)$ is thus precisely of the form (\ref{eq:dexpansion}), which shows that $(\wt\mu_1,\wt\mu_2)$ are in canonical correlations with $\rho_{\bm{n}}$ as in \eqref{eq:moment}. 
\end{proof}
\begin{rmk}
{\rm 
It is interesting to note that Theorem~\ref{thm:main} is reminiscent of a result by Griffiths and Milne (\cite{GM78}, formula (12)) for a class of dependent Poisson processes with identical marginal intensity, say $H$.  }
\end{rmk}

\begin{rmk}
{\rm 
According to Theorem~\ref{thm:main}, when $(\wt\mu_1,\wt\mu_2)$ are in canonical correlation but do not have independent joint increments, they still do have exchangeable increments. Correspondingly, in the light of Theorem \ref{thm:zeta}, for any $A\in \Xcr$, the canonical coefficients $\rho_n(A)=\Ex{Z_A^n}$ are moments of a Dirichlet random mean of the form \eqref{eq:Z_A}, except that the driving kernel $Q$ is random. The established connection with random Dirichlet means is therefore preserved in full generality. This also leads to determine a closed form expression for the canonical correlations of any vector of Gamma CRMs. Indeed, on the basis of \eqref{eq:Z_A} and of a result in \cite{Yamato} one can provide a combinatorial expansion for $\rho_{\bm{n}}$. The proof is a direct application of Proposition~2 in \cite{Yamato}, with the caveat that here we are dealing with random baseline measures.
\begin{cor}
Let $(\wt\mu_1,\wt\mu_2)$  be a vector of canonically correlated Gamma CRMs, whose marginal parameters are $(c,P_0)$. For any finite collection $\Ac=\{A_1,\ldots,A_d\}$ of pairwise disjoint sets in $\Xcr$ and $\bm{n}\in\Z_+^{d}$, the corresponding canonical correlations can be represented as
  \begin{multline}
    \label{eq:means_can_corr}
    \rho_{\bm{n}}(\Ac)=\frac{1}{\prod_{i=1}^d (cP_0(A_i))_{n_i}}\:
    \sum_{k_1=1}^{n_1}\:\cdots\:\sum_{k_d=1}^{n_d}\\[7pt]
    \E\left[\prod_{i=1}^d B_{n_i,k_i}(r_{1,i},2!r_{2,i},\ldots,
    (n_i-k_i+1)!r_{n-k+1,i})\right]
  \end{multline}
where $r_{j,i}=c\,\int_{A_i\times [0,1]}s^j\,Q(x,\ddr s)\, P_0(\ddr x)/P_0(A_i)$ and $B_{n,k}$ is, for any $k\le n$, the partial exponential Bell polynomial.
\end{cor}
}
\end{rmk}




\begin{rmk} {\rm Besides being relevant for the goals set forth in this paper, Theorem \ref{thm:zeta} and Theorem \ref{thm:main}  are further interesting in that they reproduce and complete a result on one-dimensional gamma canonically correlated random variables with infinitely divisible joint distribution, obtained by Griffiths in \cite{Bobbo(70)}. We restate Griffiths' result as follows.}

\begin{cor}
A sequence $(\rho_n)_{n\ge 1}$ defines the canonical correlations for an infinitely divisible vector $(X,Y)\in\R_+^2$ with marginally $\Gamma_{c,1}$ distributed random variables if and only if $\rho_n=\E\left[\left(\wt{p}(f)\right)^n \right]$, with $\wt p$ being a Dirichlet process with parameter measure $c P_0$ for some probability measure $P_0$ on a Polish space $(\Xc,\Xcr)$ and for some measurable function $f:\Xc\to[0,1]$ such that $\int\log(1+|f|)\,\ddr P_0<\infty$.
\end{cor}

\begin{proof} {\rm Any such infinitely divisible pair $(X,Y)$ can be interpreted as the vector of total masses $(\wt\mu_1(\Xc),\wt\mu_2(\Xc))$ where $(\wt\mu_1,\wt\mu_2)$ is a pair of canonically correlated gamma$(c,P_0)$ random measures with joint independent increments i.e. those considered in Lemma \ref{thm:zeta}. The proof thus follows immediately from Theorem \ref{thm:zeta} and Theorem \ref{thm:main}.}

\end{proof}
\end{rmk}

\section{Illustrations}\label{sec:illustr}

It is apparent from the previous results that the dependence between $\wt\mu_1$ and $\wt\mu_2$ is determined by $K$ and, more specifically, by the (possibly random) probability kernel $\{Q_x:\:x\in\Xc\}$ underlying the distribution of $K$, where $Q_x(\cdot)=Q(x,\cdot)$. We will refer to $Q_x$ as the \textit{directing kernel} for the pair $(\wt\mu_1,\wt\mu_2)$ for $K$. We will now proceed with ilustrating a few examples that correspond to different choice of $Q_x$.

\subsection{$Q_x$ degenerate at a deterministic constant, and Dawson-Watanabe processes} The simplest example one might think of corresponds to assuming 
$Q_x=\delta_z$, for any $x\in\Xc$ and 
for some constant $z\in(0,1)$. 
It is then obvious that $Z_A=z$, for every $\omega$ in $\Omega$ and $A$ in $\Xcr$. Consequently, $\rho_{\bm{n}}(\Ac)=z^{|\bm{n}|}$ for every collection $\Ac=\{A_1,\ldots,A_d\}$ of $d$ pairwise disjoint and measurable subsets of $\Xc$ and for any vector of non--negative integers $\bm{n}=(n_1,\ldots,n_d)$ with $|\bm{n}|=\sum_{i=1}^d n_i$. For $\Xc$ compact, the resulting distribution of the pair $(\wt\mu_1,\wt\mu_2)$ has the interpretation as the distribution of {$(\lambda \xi_{\lambda,0},\lambda \xi_{\lambda,t})$} for any $\lambda\in \R_+$, where $t=-2\log z/\lambda$ and $\xi_\lambda=(\xi_{\lambda,t}:t\geq0)$ is a
Dawson-Watanabe measure-valued continuous-state branching process associated to the generator
\begin{equation}
\mathcal{L}=\frac{1}{2}\int_{\Xc} \mu(dx)\frac{\delta^2}{\delta\mu(x)^2}+\frac{1}{2}c\int_{\Xc} P_0(dx)\frac{\delta}{\delta\mu(x)}-
\frac{1}{2}\lambda\int_{\Xc}\mu(dx)\frac{\delta}{\delta\mu(x)},
\label{eq:MBDIgen}
\end{equation}
 where ${\delta}/{\delta\mu(x)}$ are Gateaux derivatives. The domain $\mathcal{D}(\mathcal{L})$ of the generator is given by the space
 $$\left\{\varphi:\varphi(\mu)\equiv F\left(\mu(f_1),\ldots,\mu(f_d)\right), F\in C_c^2(\R_+^d),f_1,\ldots,f_d\in C(\Xc),d\in\N\right\},$$
 where $C(\Xc)$ is the space of all continuous functions on $\Xc$ and $C_c^2(\R_+^d)$ the space of all continuous, twice differentiable functions on $\R_+^d$ with compact support. We will refer to any such process as a $\Gamma_{c,P_0}$-DW process. This process is the measure-valued extension of the CSBI with generator \eqref{1csbi}  mentioned in Section \ref{sec:1-dim}, important in our perspective in that it constitute the extremal set in the convex class of gamma canonical correlations in one dimension (Theorem \ref{th:G69}). The joint Laplace functional transform of the $\Gamma_{c,P_0}$-DW process 
\begin{equation}
\phi_t(f,g)
=\edr^{-c\,\left[\int_{\Xc}\log(1+\frac{1}{\lambda}[f(x)+g(x)+C_{-\lambda}(t)f(x)g(x)])
\:
P_0(dx)\right]} \label{eq:dwjlt}
\end{equation}
for every $t\geq 0,$ where $C_\lambda(t):=(\edr^{\lambda t/2}-1)/\lambda$, has been determined in \cite{EG93} and extend \eqref{analyt} to infinite dimensions. Indeed one can see that, if $\lambda=1,$ \eqref{eq:dwjlt} coincides with \eqref{eq:mvexlap1}  with $z:=1-C_{-1}(t)$.
From (\ref{eq:dwjlt}), it is immediate to see that, for every $\lambda,$ the re-scaled DW process $\xi^*=\lambda \xi_\lambda$ is a process whose bivariate distributions  {$(\xi^*_s,\xi^*_{s+t})$} have, for every $s,t\geq 0,$ canonical correlations $\{z^n\}$ given by $$z=1-\lambda C_{-\lambda}(t)=\edr^{-\lambda t/2}.$$
In \cite{EG93} an expansion was derived for the transition function of the $\Gamma$-DW process that can be used straight away for all bivariate gamma measures in this class, i.e. induced by $\zeta(x)=z$. One can then state the following result as a direct application of formula {\rm (1.8)} in \cite{EG93}, where we just substitute $1-\lambda C_{-\lambda}(t)$ with $z$.
\begin{cor}\label{cor:dwalg}
Let $(\wt\mu_1,\wt\mu_2)$ be a pair of gamma $(c,P_0)$ CRMs with canonical correlations driven by a degenerate homogeneous kernel $Q_x=\delta_{z}$, where $z\in(0,1]$ is a constant. The conditional distribution of $\wt\mu_2$, given $\wt\mu_1=m_1$, can be expanded as
     \begin{multline}
    p_{\beta}(m_1,\ddr\gamma)=\sum_{n=0}^\infty \mbox{\rm Po}_{(\beta \bar{m_1})}(n)\\
    \times \,
    \int_{{\Xc}^n}P_{m_1}^n(\ddr x_1,\,\ldots\, , \ddr
    x_n)\:\Gamma_{(c+n)P_n^*,(1+\beta)^{-1}}(\ddr\gamma)
\label{eq:mvtf}
\end{multline}
where: $\beta=z/(1-z),$ $\bar{m_1}=m_1(\Xc),$ $P_{m_1}(A):=m_1(A)/\bar{m}_1,$ $$P^*_n(\ddr x)=[cP_0(\ddr x)+n\sum_{i=1}^n\delta_{x_i}(\ddr x)]/(c+n)$$
and, for every $\beta>0,$ $\Gamma_{cP_0,\beta}$ denotes the probability distribution of $\beta\mu$ where $\mu$ is a Gamma CRM with parameters $(c,P_0)$. 

\end{cor}


An algorithm for simulating from (the finite-dimensional distributions of) (\ref{eq:mvtf}) is therefore Algorithm A.3 with $b=\beta.$

\medskip


\subsection{$Q_x$ degenerate at a deterministic function}  Suppose that $Q_x=\delta_{z(x)}$, for every $x$ in $\Xc$, where $z:\Xc\to[0,1]$ is some fixed measurable function. Then $K(A)=\int_A z(x)\,\wt\mu(\ddr x)$ for any $A$ in $\Xcr$. In this case
$$Z_A\stackrel{d}{=}\int_A z(x)\, \wt p_A(\ddr x)=\int_0^1 x\ \wt{p}_{A}\circ z^{-1}(\ddr x)
$$
that is, $Z_A$ reduces to a linear functional of a Dirichlet process with \emph{deterministic parameter} $cP_0(A\cap z^{-1}(C))$ for any $C\in\Bcr([0,1])$. Hence, one has a simplification of the canonical correlations. Indeed, if one defines $r_{j,i}=c\int_{A_i} z^j(x) \,P_{0,A_i}(\ddr x)$, then
\[
  \rho_{\bm{n}}(\Ac)=\prod_{i=1}^d \frac{1}{(c P_0(A_i))_{n_i}}\:\sum_{k_i=1}^{n_i} B_{n_i,k_i}(r_{1,i},2r_{2,i},\ldots,(n_i-k_i+1)!r_{n_i-k_i+1,i})
\]
Having ascertained that $\rho_{\bm{n}}$ has a product form, one finds out that in this case the vectors $(\wt\mu_1(A_i),\wt\mu_2(A_i))$, for $i=1,\ldots,d$, are independent. Hence $(\wt\mu_1,\wt\mu_2)$ is a bivariate CRM.

\subsection{$Q_x$ degenerate at a random constant}\label{sec:ex1} Suppose $Q_x=\delta_{Z(\omega)}$ for every $x\in\Xc$ and $\omega\in\Omega$, where $Z$ is a random variable taking values in $[0,1].$ 
Hence $K=Z\,\wt\mu$. 
Under this circumstance, one has $Z_A\eqd Z$, for any $A$ in $\Xcr$, so that by virtue of Theorem~\ref{thm:main}
\[
\rho_{\bm{n}}(\Ac)=\E\left[Z^{|\bm{n}|}\right]
\]
which does not depend on $\Ac=\{A_1,\ldots,A_d\}$. Unlike the example in Section 5.1, where $Z$ is degenerate at a point $z$ in $[0,1]$, when $Z$ is random one has $$\rho_{\bm{n}}(\Ac)\ne\prod_{i=1}^d \rho_{n_i}(\{A_i\})$$
that is, the increments of $(\wt\mu_1,\wt\mu_2)$ are not independent. Nonetheless, it is easy to find a useful representation for the conditional distribution of $\wt{\mu}_2$ given $\wt{\mu}_1$, namely
$$p(m_1,\ddr\gamma)=\Ex{p_{B}(m_1,\ddr\gamma)} $$
where $B:=Z/1-Z$ and, conditional on $B=\beta$, $p_\beta$ is as in (\ref{eq:mvtf}). This also suggests a simple algorithm for generating {$\wt{\bm{\mu}}_{i,\Ac}$ conditional on $\wt{\bm{\mu}}_{j,\Ac}=\bm{x}\in\mbb{R}_+^d$, for $i,j\in\{1,2\}:i\neq j$:
\medskip


\noindent\textsc{Algorithm A.4.}
\begin{itemize}\addtolength{\itemsep}{2.5pt}
\item[\texttt{(0)}] Initialise by fixing a distribution $P_Z$ on $[0,1]$.
\item[\texttt{(i)}] Sample $Z=z$ from $P_Z$.
\item[\texttt{(ii)}] {Run Algorithm A.3. initialised by $(z/1-z),\bm{x},cP_0(\Ac)$}.
\end{itemize} 
\medskip

\noindent A particular choice of $P_Z$ leads to a more explicit formula that is pointed out below.
\begin{cor}
Suppose $Z$ has a beta distribution with parameters $(\eta c, (1-\eta)c)$, where $\eta\in(0,1)$. The canonical correlations are, then, of the form
\begin{equation}
\rho_{\bm{n}}(\Ac)=\frac{(\eta c)_{|\bm{n}|}}{(c)_{|\bm{n}|}}
\label{eq:ccbetaz}
\end{equation}
for any collection $\Ac=\{A_1,\ldots,A_d\}$ of $d$ pairwise disjoint and measurable subsets of $\Xcr$. Moreover, if $f_i:\X\to\R$ are measurable functions such that $\int\log(1+|f|)\ddr P_0<\infty$, for any $i=1,2$, the Laplace functional transform of $(\wt\mu_1,\wt\mu_2)$ evaluated at $(f_1,f_2)$ is
\begin{equation}
\phi(f,g)=\phi_{(cP_0)}(f_1)\,\phi_{(cP_0)}(f_2)\:
\E\left[\left(1+\int \theta\:\ddr\wt p_{cP_0}\right)^{-\eta c}\right]
\label{eq:mixlf}
\end{equation}
where $\wt p_{cP_0}$ is a Dirichlet process with parameter measure $cP_0$ and $\theta=f_1f_2(1+f_1)^{-1}(1+f_2)^{-1}$.
\end{cor}

\begin{proof}
The form of the canonical correlations in \eqref{eq:ccbetaz} follows from Theorem 6 the fact that $Z_i=Z$ for any $i=1,\ldots,d$. Moreover, given $K=Z\wt\mu$ one has
\[
\E\left[\edr^{\int \theta\:\ddr K}\right]=\E\left[\frac{1}{(1-Z_\eta\,\int\theta\:\ddr\wt p)^c}\right]=\E\left[\frac{1}{(1-\int\theta\:\ddr\wt p)^{\eta c}}\right]
\]
where the first equality follows from the Markov--Krein identity for random Dirichlet means, and the last one is a well-known hypergeometric identity that can be easily recovered directly by expanding the joint Laplace transform of $(\wt{\mu}_1,\wt{\mu}_2)$. This shows \eqref{eq:mixlf}. 
\end{proof}


In \eqref{eq:mixlf} one notices that the expectation on the right--hand side is the generalised Cauchy--Stieltjes transform of the mean $\int \theta\ddr\wt p_{cP_0}$ of a Dirichlet process. Note that if $f_1$ and $f_2$ are simple functions taking on a finite number of values, i.e. $f_1=\sum_{i=1}^{d}s_i\indic_{A_i}$, $f_2=\sum_{i=1}^{d}t_i\indic_{A_i}$ for a collection $A_1,\ldots A_d,$ of pairwise disjoint sets in $\Xcr$, such a transform can be expressed in terms of the fourth Lauricella hypergeometric function $F_D$. Indeed, upon setting $\theta_i=s_it_i(1+s_i)^{-1}(1+t_i)^{-1}$ and  $\alpha_i=cP_0(A_i)$, for $i=1,\ldots,d$, one has
\[
\frac{\phi(f_1,f_2)}{\phi_{(cP_0)}(f_1)\phi_{(cP_0)}(f_2)}=
\frac{\prod_{i=1}^d\Gamma(\alpha_i)}{\Gamma (|\bm{\alpha}|)}
F_D\left({\eta c},\alpha_1,\ldots,\alpha_d;{c};\theta_1,\ldots,\theta_d\right)
\]
where $|\bm{\alpha}|=\sum_{i=1}^d\alpha_i$. 

\subsection{Model with random elements in common} 

{In the previous subsections we have been considering a broad class of examples of canonically correlated gamma CRMs whose directing kernel is of the form
$$
Q_x(\ddr s)=\delta_{\zeta(x,\omega)}(\ddr s) \qquad\forall x\in\Xc \quad \forall\omega\in\Omega,
$$
implying that the $s$ component is degenerate at a measurable function $\zeta:\Xc\times\Omega \to [0,1]$. It, then, follows that $K(f)\eqd \int_{\Xc} \,f(x)\zeta(x)\:\wt{\mu}^*(\ddr x)$, where $\wt{\mu}^*$ is a Gamma CRM with parameter $(c,P_0)$. Equivalently, $K$ is a weighted gamma CRM on $(\Xc,\Xcr)$, conditional on $\zeta$. The corresponding bivariate gamma CRMs have thus a joint Laplace transform of the form
\begin{equation}
\phi(f,g)=\,\phi_{cP_0}(f)\:\phi_{cP_0}(g)\:
\E\left[\edr^{c\int\log\left(1+\zeta\,\frac{fg}{(1+f)(1+g)}\right)\,\ddr P_0}\right].
\label{eq:mvexlap1}
\end{equation}
The canonical correlation sequences could have been derived from an expansion of (\ref{eq:mvexlap1}) with both $f$ and $g$ as simple functions. In particular, by Theorem~\ref{thm:main}, they were determined by joint moments of random variables of the form: 
\begin{equation}
Z_{A}\overset d = 
\int_{A}\zeta(x,\omega)\:\wt{p}_{A}(\ddr x), \qquad i=1,\ldots,d
\label{egcc}
\end{equation}
where $\wt{p}_{A}$ is a Dirichlet process with parameter measure $cP_0$ restricted to set $A$.}

{A different structure for the directing kernel needs to be identified if one wants to construct vectors of CRMs being in canonical correlation in the spirit of \cite{GM78}, where the authors study a class of Poisson random measures $(\wt{J}_1,\wt{J}_2)$ such that $\wt{J}_i=J_i+J_0$ where $J_1$, $J_2$ and $J_0$ are independent Poisson random measures with $J_1\stackrel{d}{=}J_2$. Dependence is, thus, induced by a common source of randomness in $J_0$. A finite-dimensional version of this model was proposed in \cite{KP00} where the connection with canonical correlations was studied for a large class of marginals including multivariate Poisson and multivariate Gamma. The particular case where $J_0$ has intensity measure $\eta cP_0$ and $J_i$, for $i=1,2$, have intensity $(1-\eta)cP_0$ $(\eta\in[0,1])$ was employed by \cite{Nipoti} to build dependent nonparametric priors for inference on partially exchangeable data. In our measure-valued setting, with Gamma marginals, the latter model can be viewed as corresponding to a directing kernel of the type
\begin{equation}
Q_x(\ddr s)=Q^\eta(\ddr s):=\eta\ \delta_{\{0\}}(\ddr s)+(1-\eta)\delta_{\{1\}}(\ddr s), 
\label{reckernel}
\end{equation}
for any $x\in\Xc$ and for some $\eta\in(0,1)$. 
In this case it can be easily seen that $K$ is a Gamma CRM with parameters $(c(1-\eta), \, P_0)$. Finally, recall that if $\wt p$ is a Dirichlet process on $[0,1]$ parameter $(c, Q^\eta)$, the random variable $Z=\int_0^1 z\ \wt p(\ddr z)$ has distribution Beta$(c\eta,(1-\eta)c)$ (see \cite{LPsurvmeans}), thus yielding}
\begin{cor}
Let $\mu_1$, $\mu_2$ and $\mu_0$ be independent $\Gamma_{(1-\eta)cP_0}$, $\Gamma_{(1-\eta)cP_0}$ and $\Gamma_{\eta cP_0}$ CRMs, respectively. Consider a vector $(\wt\mu_1,\wt\mu_2)$ of gamma CRMs in canonical correlation with directing kernel $Q^\eta$. 
 Then the distribution of $(\wt\mu_1,\wt\mu_2)$ is the model with random elements in common
  \begin{equation}
    \wt\mu_1
    =\mu_1+\mu_0 \qquad\qquad
    \wt\mu_2=
    \mu_2+\mu_0
    \label{recmodel}
  \end{equation}
and the canonical correlations are
\begin{equation}
  \label{eq:rec.cc}
  \rho_{\bm{n}}(\Ac)=\prod_{i=1}^d \frac{(\eta c P_0(A_i))_{(n_i)}}{(c P_0(A_i))_{(n_i)}}=\prod_{i=1}^d\Ex{Z_{A_i}^{n_i}},
\end{equation}
where $(Z_{A_i}:i=1,\ldots,d)$ are independent random variables, respectively with beta $(\eta c P_0(A_i),\, (1-\eta)cP_0(A_i))$ distribution, for $i=1,\ldots,d$, for any collection $\Ac=\{A_1,\ldots,A_d\}$ of disjoint and measurable subsets of $\Xc$. Thus the process has jointly independent increments.
\end{cor}

\begin{proof}
The law of a pair $(\wt\mu_1,\wt\mu_2)$ directed by $Q^\eta$ and the law of a pair constructed as in \eqref{recmodel} have the same Laplace functional with identical driving CRM $K$. In order to identify the canonical correlations we determine an expansion of the Laplace transform of $K$ evaluated at a simple function $f=\sum_{i=1}^d s_i\indic_{A_i}$ with sets $A_i\in\Xcr$ being pairwise disjoint. It can then be seen that
\begin{align*}
\E\left[\edr^{K(f)}\right]
&=\sum_{\bm{n}\in\Z_+^d}\E\left[\prod_{i=1}^d
\frac{K^{n_i}(A_i)}{n_i!}s_i^{n_i}
\right]\\
&=
\sum_{\bm{n}\in\Z_+^d}\prod_{i=1}^d (cP_0(A_i))_{n_i}\:
\frac{(\eta cP_0(A_i))_{(n_i)}}{(cP_0(A_i))_{(n_i)}}\:\frac{s_i^{n_i}}{n_i!}
\end{align*}
and this implies that the canonical correlations are as in \eqref{eq:rec.cc}. Hence the $Z_{A_i}$'s are independent beta random variables with respective parameters $(\eta cP_0(A_i),\,(1-\eta) cP_0(A_i))$ and the proof is complete.
\end{proof}

\section{Measure--valued Markov processes with canonical correlations}\label{sec:markovcc}
In this Section we apply Theorem \ref{thm:main} to characterise the class of time-homogeneous, reversible, Feller transition functions $\{P_t:\:t\geq 0\}$ defined on $M_\Xc \times \Mcr_\Xc,$ with $\Gamma_{cP_0}$ reversible (stationary) measure such that, for every $\mu\in M_\Xc$ and $t\geq 0,$ all the finite dimensional distributions associated to $P_t(\mu,\cdot)$ have multivariate Laguerre polynomial eigenfunctions. In other words, for every $t>0$, the pair $(\xi_0,\xi_t)$ will be distributed as a pair of  $\Gamma_{cP_0}$ random measures in canonical correlation. We will then say that $\{\xi_t:\:t\geq 0\}$ has \emph{canonical autocorrelation} with respect to its gamma$(c,P_0)$ reversible measure. We prove that the law of any such process corresponds to the law of a Dawson-Watanabe process with generator (\ref{eq:MBDIgen}), time-changed by means of some independent, \emph{one-dimensional} subordinator. From now on we will assume, for simplicity, that $\Xc$ is a compact metric space. 

We will derive our characterisation in two steps. First, we fix a collection of $d$ pairwise disjoint Borel sets $\Ac=\{A_1,\ldots,A_d\}$ and correspondingly set $\bm{\alpha}=(cP_0(A_1),\ldots,cP_0(A_d))$. We will establish (Section \ref{sec:step1}) necessary and sufficient conditions for a time-homogeneous Markov process $\{\bm{X}_t^\Ac:\:t\ge0\}$  to be in \emph{gamma-canonical autocorrelation} i.e. to have $\Gamma_{\bm{\alpha},1}=\times_{i=1}^d\Gamma_{\alpha_i,1}$ as stationary measure and to satisfy, for every $t\geq 0,$
\begin{equation}
\Ex{\wt{L}_{\bm{\alpha},\bm{n}}(\bm{X}_0)\wt{L}_{\bm{\alpha},\bm{m}}(\bm{X}_t)}
=\delta_{\bm{m}\bm{n}}\:\rho_{\bm{n}}(t)\qquad \bm{m}, \bm{n}\in\mbb{Z}_+^d
\label{eq:lag_markov}
\end{equation}
for some sequence of functions $\{\rho_{\bm{n}}(t):\bm{n}\in\mbb{Z}_+^d,t\geq 0\}$ guaranteeing the Markov property. We will prove that all such Markov vectors can be derived via a time-change of a $d$-dimensional continuous-state Branching process with immigration [ref to model], by means of an independent, $d$-dimensional subordinator (see later, Definition \ref{def:dsub}).\\
As a second step, in Section \ref{sec:step2} we will look for conditions guaranteeing that, for a measure-valued process $\{\xi_t:\:t\geq 0\}$, both Markovianity and \eqref{eq:lag_markov} hold consistently for all its finite-dimensional projections $\xi_t(\Ac):=(\xi_t(A_1),\ldots,\xi_t(A_d))$ as we let $d$ and $A_1,\ldots,A_d\in\Xcr^d$} vary. It is just such consistency condition that will lead us to conclude that, at the infinite-dimensional level, all canonically auto-correlated Gamma measure-valued processes can be obtained by time-changing a Dawson-Watanabe process merely by means of one-dimensional subordinators.


 \subsection{Step 1. Markov gamma vectors}\label{sec:step1} The approach we use is similar to the method suggested in Bochner \citep{Bo54} to study Gegenbauer processes on $[-1,1],$ and employed by Griffiths \cite{Bobbo(69)} to characterise one-dimensional Laguerre stochastic processes. Assume $\{\bm{X}_t:\: t\ge 0\}$ is a $\R_+^d$-valued time-homogeneous Markov process with canonical autocorrelation and with $\Gamma_{\bm{\alpha},1}=\times_{i=1}^d\Gamma_{\alpha_i,1}$ as stationary measure, where $\bm{\alpha}=(\alpha_1,\ldots,\alpha_d)\in(0,\infty)^d$. Consider the associated conditional expectation operator
 $$ P_{t}f(\bm{x})=\Ex{f(\bm{X}_t)\mid \bm{X}_0=\bm{x}}\qquad t\geq 0$$
for every bounded measurable functions $f:\R_+^d\to \R$. Note that, in order for $P_t$ to correspond to a Markov transition kernel, $t\mapsto P_t$ must be continuous in $t$ and $P_0$ must be the identity operator. From \eqref{eq:cexp} one has
$$P_{t}\wt{L}_{\bm{n},\bm{\alpha}}(\bm{x})=\rho_{\bm{n}}(t)\:\wt{L}_{\bm{n},\bm{\alpha}}(\bm{x}) \qquad \forall t\geq 0.$$
The conditions above and the semigroup (Chapman-Kolmogorov) property
$P_{t+s}=P_t P_s$, for all non negative $s$ and $t$, are satisfied if and only if, for every $\bm{n}\in\Z_+^d$,
  \begin{align}
& \rho_{\bm{n}}(t+s)=\rho_{\bm{n}}(t)\,\rho_{\bm{n}}(s)
\qquad \forall s,t\geq
0
\label{markovcc1}\\
&\rho_{\bm{n}}(0)=1\\
&  t\mapsto\rho_{\bm{n}}(t)\qquad \mbox{ continuous}
  \label{markovcc4}
  \end{align}
These properties define an ordinary kinetic equation for $t\mapsto \rho_{\bm{n}}(t)$ for any $\bm{n}$, that is: (\ref{markovcc1})-(\ref{markovcc4}) hold if and only if there exists a sequence of (non-negative) reals $\lambda_{\bm{n}}$ such that
 $$\rho_{\bm{n}}(t)=\edr^{-t\lambda_{\bm{n}}} \qquad \forall t\geq 0.$$
So long as the conditions above are satisfied, then $\{P_t:\: t\ge 0\}$ is a Feller semigroup  in that $P_t$ maps the space $L_2(\R_+^d,\Gamma_{\bm{\alpha},1}) $ onto itself (being the Laguerre polynomials an orthogonal basis for it) and, in particular, it maps the space of all bounded continuous functions vanishing at infinity onto itself.

In the light of this discussion, we aim at characterising all  possible multi-indexed sequences $\{\lambda_{\bm{n}}: \:\bm{n}\in\Z_+^d\}$ that may serve as ``rates'' for the canonical autocorrelations of $\bm{X}$. It turns out that each such sequence of rates is given by the Laplace exponent, evaluated at $\bm{n}\in\Z_+^d,$ of some multivariate subordinator.

\begin{defi} \label{def:dsub}{\cite{BNPS01}}
A $\R_+^d$-valued multivariate subordinator $\bm{Y}=\{\bm{Y}_t:\: t\ge 0\}$ (starting at $(0,\ldots,0)$)
is a stochastic process whose Laplace transform admits, for every $t$, a L\'evy-Kintchine representation of the form
\begin{equation}
-\frac{1}{t}\log\Ex{\edr^{-\sum_{i=1}^d\phi_i Y_{i,t}}}=\sum_{i=1}^d\phi_i c_i+\int_{\R_+^d}
\left(1-\edr^{-\sum_{i=1}^d\phi _i u_i}\right)\eta(\ddr\bm{u})
\label{dsub}
\end{equation}
for any $(\phi_1,\ldots,\phi_d)\in\R_+^d$, from some constants $c_i\geq 0,\ i=1,\ldots,d$ and L\'{e}vy measure $\eta$ on $\R_+^d$ concentrated on $(0,\infty)^d$.
\end{defi}

Note that $\eta$ must be such that
$\int_{\R_+^d}(\|\bm{u}\|\wedge 1)\eta(\ddr\bm{u})<\infty$, with $\|\bm{u}\|$ denoting the Euclidean norm (\cite{BNPS01}, Theorem 3.1). This implies that $\bm{Y}_t$ has, for every $t\geq 0,$ a multivariate infinitely divisible law. Henceforth, we shall denote the \textit{Laplace exponent} of $\bm{Y}$, i.e. the right--hand side of \eqref{dsub}, as $\psi_{\bm c,\eta}$, omitting the subscript when there is no risk of confusion. A consequence of (\ref{markovcc1})-(\ref{markovcc4}) is the following
\begin{thm}\label{pr:dmarkov}
 For every $d\ge 1$ and every collection $\bm\alpha=(\alpha_1,\ldots,\alpha_d)\in(0,\infty)^d$, a time-homogeneous  Markov process $\{\bm{X}_t:\: t\geq 0\}$ in $\R_+^d,$ with $\Gamma_{\bm{\alpha},1}$ 
stationary distribution, is time-reversible with canonical autocorrelation function $\rho$ if and only if there is a multivariate subordinator $\bm{S}=\{\bm{S}_t:\:t\geq 0\}$ with corresponding L\'{e}vy exponent $\psi$ such that 
\begin{equation}
\left(\bm{X}_t:\: t\geq 0\right)\overset d = (\bm{X}^{DW}_{\bm{S}_t}:\: t\geq 0)\label{gseqd}
\end{equation}
for every $\bm{n}\in\Z_+^d,$ and $t\in\R_+,$ where $(\bm{X}^{DW}_t:t\geq 0)$ is a continuous-state branching process with transition function \eqref{dCSBI_tf} with parameter $\bm{\alpha}$. The corresponding canonical correlation coefficients are of the form:
 \begin{equation}
\rho_{\bm{n}}(t)=\edr^{-t\psi(\bm{n})}.
\label{rhont}
\end{equation}
 \end{thm}

 \begin{proof}
 If $\bm{X}$ has canonical autocorrelations, then Theorem \ref{thm:main} and (\ref{markovcc1})-(\ref{markovcc4}) imply that for every $t$, on one hand, there exists a random vector $\bm{Z}_t=(Z_{1,t},\ldots,Z_{d,t})\in [0,1]^d$ such that
 $$\rho_{\bm{n}}(t)=\Ex{Z_{1,t}^{n_1}\cdots Z_{d,t}^{n_d}};$$
 on the other hand, that
 $$\rho_{\bm{n}}(t)=e^{-\lambda_{\bm{n}} t}$$ for a constant $\lambda_{\bm{n}}.$
If $H_t$ is the law of $\bm{Z}_t$, then
$$\frac{\rho_{\bm{n}}(t)-\rho_{\bm{n}}(0)}{t}=
\int_{[0,1]^d}\left(\prod_{i=1}^dz_i^{n_i}-1\right)\frac{H_t(\ddr z_1,\ldots,\ddr z_d)}{t}.$$
and by definition of exponential function,
$$\lambda_{\bm{n}}=\lim_{t\downarrow 0}\frac{\rho_{\bm{n}}(0)-\rho_{\bm{n}}(t)}{t}.$$
 For every $t$ consider the measure $G\left(\cdot,t\right)$ defined by
$$G\left(B,t\right):=\int_B\left(1-\prod_{i=1}^d z_i\right)\frac{H_t(\ddr z_1,\ldots,\ddr z_d)}{t},$$ for every Borel set B of $[0,1]^d.$ 
Since $\rho_{\bm{n}}(0)=1$ for every $\bm{n},$  $G\left([0,1]^d,t\right)=t^{-1}\left(1-\edr^{-t\lambda_{\bm{1}}}\right) ,$ where $\bm{1}=(1,1,\ldots,1).$ Thus $G(\cdot, t)$ has, for every $t\leq \delta,$ say, a total mass that is uniformly bounded by a constant that only depends on $\delta.$ 
Thus, for every set $I(\bm{y}):=\otimes_{i=1}^d [0,y_i]$, the function $\bm y\mapsto G(I(\bm{y}),t)$ is a uniformly bounded function, non-decreasing in its arguments. It is therefore possible (Helly's first theorem, see \cite{L75}, Theorem 1.3.9) to choose a sequence $t_k\downarrow 0$ for which $G(I(\bm{y}), t_k)$ converges weakly, as $t\to 0$, to a function $\bm{y}\mapsto G(I(\bm{y}))$ non-decreasing in all its arguments. By continuity on $[0,1]^d$ of the function
$$ \frac{1-\prod_{i=1}^dz_i^{n_i}}{1-\prod_{i=1}^dz_i} ,$$ we can thus rewrite (Helly's second Theorem, see \cite{L75}, Theorem 3.1.10),
$$\lambda_{\bm{n}}=\lim_{t\downarrow 0} \int_{[0,1]^d}\left(\frac{1-\prod_{i=1}^dz_i^{n_i}}{{1-\prod_{i=1}^dz_i}}\right) 
G(\ddr \bm{z},t)= \int_{[0,1]^d}\left(\frac{1-\prod_{i=1}^dz_i^{n_i}}{{1-\prod_{i=1}^dz_i}}\right)G(\ddr\bm{z})$$

Denoted $H(\ddr\bm{z}):=\left(1-\prod_{i=1}^dz_i\right)G(\ddr z_1\cdots \ddr z_i),$ we see that
$$
\lambda_{\bm{n}}=\int_{[0,1]^d}\left(1-\prod_{i=1}^dz_i^{n_i}\right)H(\ddr\bm{z})
$$
The change of variable $s_i=-\log z_i $, along with the corresponding change of measure $\wt{G}(B)=G\left(\{\bm{z}\in[0,1]^d:\: (-\log z_1,\ldots,-\log z_d)\in B\}\right)$ for any $B$ in $\Bcr(\R_+^d)$, leads to
 \begin{equation}
\rho_{\bm{n}}(t)=
\exp\left\{-t\int_{\R_+^d}
\left(1-\edr^{\sum_{i=1}^dn_is_i}\right)\wt{G}(\ddr\bm{s})\right\}.
\label{lkrho}
\end{equation}
This, in turn, implies that $$\bm{S}=\left\{\left(-\log Z_{1,t},\ldots,-\log Z_{d,t}\right): \: t\geq 0\right\}
$$ is a $d$-dimensional subordinator starting at the origin with L\'{e}vy measure $\wt{G}.$ We have just seen that, for every $\bm n$ and $t$, 
$$\rho_{\bm{n}}(t)=\Ex{e^{-\sum_{i=1}^dn_i S_{i,t}}}$$
for a $d$-dimensional subordinator with L\'evy intensity $\wt G$ hence it is of the form \eqref{rhont}. The identity in distribution \eqref{gseqd} follows. Indeed,
consider a $d$-dimensional continous-state branching process $(\bm{X}^{DW}_t:t\geq 0)$, with parameters $\bm{\alpha}$, for simplicity, with time rescaled by a factor of 2. We have seen in Section \ref{sec:dCSBI} (see \eqref{dCSBI_tf}) that it has $\Gamma_{\bm\alpha,1}$ stationary distribution and canonical autocorrelations
$$\rho_{\bm{n}}( t)=e^{- |\bm n |t},$$ (with $t$ instead of $t/2$ thanks to the time-rescaling) and its transition function can be generated from Algorithm A.3.
Therefore, for every $L^2(\R_+^d,\Gamma_{\bm\alpha,1})$ function, 
\begin{align}
P^{DW}_t f(\bm x)&=\sum_{\bm{n}}e^{- |\bm n| t} \wh{f}(\bm n) L_{\bm n, {\bm\alpha}}(\bm x)\notag\\
&=\sum_{\bm{n}}\left[\prod_{i=1}^d e^{-n_i t}\right]\wh{f}(\bm n) L_{\bm n, {\bm\alpha}}(\bm x).
\label{eq:dprodtf}
\end{align}
Consider the subordinated process $\wt{\bm X}_t=\bm{X}^{DW}_{\bm S_t}$, where $(\bm S_t)$ has L\'evy  intensity $\wt G$ and is independent of $\bm{X}^{DW}.$ Then
\begin{align}
\Ex{f(\wt{\bm{X}}_{t})\mid \wt{\bm{X}}_{0}=\bm x}&=
\Ex{f(\bm{X}^{DW}_{\bm S_t})\mid \bm{X}^{DW}_{0}=\bm x}\notag\\
&=\sum_{\bm{n}}\Ex{e^{-\sum_{i=1}^d |\bm n| S_{i,t} }} \wh{f}(\bm n) L_{\bm n, {\bm\alpha}}(\bm x)\notag\\
&=\sum_{\bm{n}}\rho_{\bm{n}}(t) \wh{f}(\bm n) L_{\bm n, {\bm\alpha}}(\bm x)\label{eq:subdwrho}
\end{align}
where $\rho_{\bm{n}}(t)$ is given by \eqref{lkrho}. 
In other words,  $(\wt{\bm{X}}_t)$ has the same transition semigroup as $(\bm X_t)$, thus the laws of the two processes coincide which proves the necessity part of the theorem.\\
The equality \eqref{eq:subdwrho} also immediately proves the sufficiency part as well: time-changing a d-dimensional CSBI process by an independent subordinator will give a new Markov process with canonical autocorrelations of the form \eqref{lkrho}.
\end{proof}

\subsubsection{Discrete-time construction.} To give an heuristic interpretation to the Theorem (especially the sufficency part), one can consider the following \lq\lq Poisson continuous-time embedding\rq\rq of the following discrete-time construction (similar one-dimensional constructions have been proposed e.g. in \cite{Bo54}, \cite{Bobbo(69)} and \cite{GS10}).
Consider a Markovian sequence $\bm{X}_0,\bm{X}_1,\ldots$ of $d$--dimensional random vectors such that for every $j=0,1,\ldots,$ $(\bm{X}_j,\bm{X}_{j+1})$ is a bivariate gamma$(\bm\alpha,1)$ random pair of vectors with canonical correlations $\rho_{\bm{n}}(1)=\int_{[0,1]^d}\prod_{i=1}^d z_i^{n_i}\:F(\ddr \bm{z})$ for some measure $F.$
It is easy to see that, from the symmetry of the one-step kernel, the chain is reversible and that 
 $$\Ex{\wt{L}_{\bm{n},\bm{\alpha}}(\bm{X}_{k})\mid \bm{X}_0=\bm{x}}=\rho_{\bm{n}}^k(1)\:\wt{L}_{\bm{n},\bm{\alpha}}(\bm{x}),\qquad k=0,1,\ldots$$
that is $(\bm{X}_0,\bm{X}_k)$ have canonical correlations $\rho_{\bm{n}}(k)=\rho_{\bm{n}}^k(1)$ for any  $k$.
Consider now a Poisson process $\{N_t:\: t\geq 0)\}$ on $\R_+$ with rate $\gamma,$ independent of $(\bm{X}_k)_{k\ge 0}$ and set
$$\tilde{\bm{X}}(t):=\bm{X}_{N(t)},\qquad  t\geq 0.$$
By construction, $\{\tilde{\bm{X}}_t:\:t\geq 0\}$ is again a Markov process and has canonical autocorrelations
$$
\rho_{\bm{n}}(t)
=\Ex{\rho_{\bm{n}}^{N(t)}}=\exp\left\{-t\gamma\left[1-\rho_{\bm{n}}(1)
\right]\right\}
$$
But this can also be written as
$$\gamma\left[1-\rho_{\bm{n}}(1)\right]
=\int_{[0,1]^d}\left(1-\prod_{i=1}^dz_i^{n_i}\right){H}_\gamma(\ddr\bm{z})
$$
where $H_\gamma:=\gamma F$. This property is preserved even as $\gamma\to\infty$, if $\gamma$ and $F$ are chosen so that the limit measure (with respect to the topology of the weak convergence) $H:=\lim_{\gamma\to \infty}\gamma F$ exists, so that
$$\gamma[1-\rho_{\bm{n}}(1)]=\int_{[0,1]^d}\left(1-\prod_{i=1}^dz_i^{n_i}\right){H}_\gamma(\ddr\bm{z})\to \int_{[0,1]^d}\left(1-\prod_{i=1}^dz_i^{n_i}\right){H}(\ddr\bm{z}) .$$
So, after the change of variable $s_i=-\log z_i,$ denoting with $\tilde{G}^{*}$ is the corresponding induced measure on $\R_+^d,$
$$\rho_{\bm{n}} (t)=\exp\left\{-t\int_{\R_+^d}
\left(1-\edr^{-\sum_1^dn_is_i}\right)\tilde{G}^{*}(\ddr\bm{s})\right\}
$$
is, for every $t,$ a canonical correlation sequence and, since it satisfies (\ref{markovcc1})-(\ref{markovcc4}), it is a Markov canonical correlation sequence, and it is in a (multivariate) L\'evy-Kintchine form.

 \subsection{Step 2. Markov gamma measure-valued processes}\label{sec:step2}
We are now ready to derive the characterisation of measure-valued gamma processes we are seeking, as Dawson-Watanabe processes time-changed by a one-dimensional subordinator. From Section~5.1, we know that a DW process $\xi=\{\xi_t:\: t\geq 0\}$ with (sub)criticality parameter $\lambda=1$ has canonical correlations
  $$\rho_{\bm{n}}(\Ac,t)=\edr^{-|\bm{n}|t/2}$$
for every $\bm{n}\in\Z_+^d$, $d\in\N$, $t\geq 0$ and $\Ac=\{A_1,\ldots,A_d\}.$
The exponent does not depend on $\Ac$. Moreover, it corresponds to the Laplace exponent of a pure drift one-dimensional subordinator 
and, thus, all its finite-dimensional distributions on disjoint subsets are of the form (\ref{rhont}).
 Now take any (one-dimensional) subordinator $S=\{S_t:\: t\ge 0\}$, driven by a L\'{e}vy measure $\nu$, defined on the same probability space as and independent of $\xi$. Denote with $\psi_\nu$ its L\'{e}vy exponent and define the process $\xi^S$ as $$\xi^S:=\left\{\xi_t^S:\: t\ge 0\right\}\overset d = \left\{\xi_{S_t}:t\geq 0\right\}$$
The process $\xi^S$ is still a Markov process, since every time-change via a subordinator maps any Markov process to a new Markov process (see e.g. \cite{Sato99}, Chapter 6). The stationary measure is unchanged 
so that, if $\{P^{\mbox{\scriptsize{DW}}}_t:\: t\ge 0\}$ and $\{P_t^S:t\ge 0\}$ denote the semigroup corresponding to the transition function of $\xi$ and $\xi^S$ respectively, then, for every bounded test function $\varphi=\varphi(\mu)$ in the domain $\mathcal{D}(\mathcal{L})$ of the generator of $(P^{\mbox{\scriptsize{DW}}}_t:t\geq 0)$,
$$\Ex{P_t^{S}\varphi({\mu})}=\Ex{\E_{cP_o}\left[P_{S_t}^{\mbox{\scriptsize{DW}}}
\varphi({\mu})\right]}=\E_{cP_0}{\varphi({\mu})},$$
where $\E_{cP_0}$ indicated the expectation with respect to the DW stationary measure $\Gamma_{cP_0}$. 
It can be seen that the process of the finite-dimensional distributions of $\xi^S$ obeys to
\[
P^S_t \wt{L}_{\bm{n},\bm{\alpha}}(\bm{x})
=\Ex{\edr^{-|\bm{n}|S_t/2}}\wt{L}_{\bm{n},\bm{\alpha}}(\bm{x})
=\edr^{-t\psi_\nu(|\bm{n}|/2)}\wt{L}_{\bm{n},\bm{\alpha}}(\bm{x})
\]
for every $\bm{n}\in\Z_+^d$ and $\Ac\in\Xcr^*$ such that $cP_0(\Ac)=\bm{\alpha}.$ Consequently  $\xi^S$ has Markov canonical correlation coefficients satisfying (\ref{rhont}) and, thus, subordinated $\Gamma$-DW processes have canonical auto-correlation. Moreover, they are Feller processes since $\xi$ is a Feller process. We have just proved the sufficiency part of the next Theorem. A less immediate task is to prove the necessity part.

 \begin{thm}
\label{thm:markovcc}
A time-reversible, time-homogeneous Feller process $\xi=\{\xi_t:\: t\ge 0\}$ with values in $(M_\Xc,\Mcr_\Xc)$ and stationary measure $\Gamma_{cP_0}$, has canonical autocorrelations if and only if the law of $(\xi_t:t\geq 0)$ is equal to the law of $(\xi^{DW}_{S_t}:t\geq 0)$ where $\xi^{DW}=(\xi^{DW}_t:t\geq 0)$ is a Dawson-Watanabe process with $\Gamma_{cP_0}$ stationary distribution and criticality parameter $\lambda=1$, and  $S=(S_t:t\geq 0)$ is a one-dimensional subordinator independent of $\xi^{DW}$. The generator of $\xi$ is thus given by
$$\mathcal{L}^S\varphi (\mu)=\mathcal{L}^{DW}\varphi(\mu) +\int_0^\infty P^{DW}_s\varphi(\mu)\ \nu(\ddr s)$$
where $\mathcal{L}^{DW}$ and $(P^{\mbox{\scriptsize{DW}}}_t:t\geq 0)$ are-respectively, the generator \eqref{eq:MBDIgen} and the semigroup of the DW diffusion $\xi^{DW}$ (see Section \ref{sec:ex1}) and $\nu$ is the L\'{e}vy measure of the subordinator $S$. $\mathcal{D}(\mathcal{L})$ is a core for the domain of $\mathcal{L}^S.$\\
\end{thm}

\begin{proof}
Suppose $\xi=\{\xi_t:\: t\geq 0\}$ is a time-homogeneous, reversible, measure--valued Feller process with $\Gamma_{cP_0}$ stationary measure and canonical autocorrelations $\rho$. Then  by Proposition 2, for every $t,$ every $d$ and every $\Ac=(A_1,\dots,A_d)\in\Xcr^*$, the canonical correlations of $(\xi_0(A_1),\ldots,\xi_0(A_d))$ and $(\xi_t(A_1),\ldots,\xi_t(A_d))$, must satisfy
\begin{equation}
\rho_{\bm{n}}(\Ac,t)=\edr^{-t\psi_\Ac(\bm{n})}=\Ex{\edr^{-\sum_{i=1}^dn_iS_{A_i,t}}},\qquad \bm{n}\in\Z_+^d\end{equation}
for a multivariate subordinator $\bm{S}^\Ac=((S_{A_1,t},\ldots,S_{A_d,t}):t\geq 0)$ with Laplace exponent $\psi_\Ac.$ Our goal is to prove that, for every $t$, every $\Ac$ and every $d,$  
\begin{equation}
S_{A_1,t}=\ldots=S_{A_d,t}=S_{t}
\end{equation}
almost surely, for a unique subordinator $(S_t:t\geq 0).$ This will prove the claim for then
$$\rho_{\bm{n}}(\Ac,t)=\Ex{\edr^{-|\bm{n}|S_{ t}}}=\edr^{-t\psi(|\bm n|)}$$ where $\psi$ is the L\'evy exponent of $(S_t)$, and this  is precisely the form of the canonical correlation coefficients of a Dawson-Watanabe process time-changed by a one-dimensional subordination. 
%
Now, from Lemma \ref{l:zadditive}, for every $A\in\Xcr$ and every measurable partition $A_0,A_1$ of $A$, we know that the canonical correlation coefficients must obey to
\begin{eqnarray}
\rho_{n}({A},t)&=&\rho_n(A_0\cup A_1,t)\notag\\
&=&\sum_{j=0}^n{n\choose j}
\frac{(cP_0(A_0))_j(cP_0(A_1))_{n-j}}{\left(cP_0(A)\right)_n}\rho_{j,n-j}(A_0,A_1,t)\notag\\
&=&\Ex{\rho_{J,n-J}(A_0,A_1,t)}
\label{psiproof1}
\end{eqnarray}
where $J$ is a random variable whose distribution is beta--binomial distribution with parameter $(n;cP_0(A_0),cP_0(A_1))$ and we can interpret
$$\rho_{j,n-j}(A_0,A_1,t)=\Ex{\edr^{-JS_{A_0,t}-(n-J)S_{A_1,t}}\mid J=j}.$$
where $J$ and $(S_{A_0,t},S_{A_1,t})$ are independent. But by virtue of (\ref{dsub}), necessarily $\rho_{n}(\Ac,t)=\rho_{n}^t(\Ac,1)$ for any $t\ge 0$, which entails
$$
\rho_{n}^t(A,1)=\Ex{\rho_{J,n-J}^t(A_0,A_1,1)}\qquad \forall t\ge 0.
$$
This equality for $t=2$, shows that 
$$\E^2\left\{\rho_{J,n-J}(A_0,A_1,1)\right\}=\Ex{\rho_{J,n-J}^2(A_0,A_1,1)}$$ i.e.
the random variable 
$$\rho_{J,n-J}(A_0,A_1,1)=\Ex{\edr^{-JS_{A_0,1}-(n-J)S_{A_1,1}}\,\mid\, J}$$
has variance $0$ whence, almost surely, $S_{A_0,1}=S_{A_1,1}$. By Lemma \ref{l:zadditive} this must hold true for every choice of $A$ and of a partition $\{A_0,A_1\}$ of $A.$ This argument also leads to state that $S_{A,1}=S_{A^c,1}=S_{\X,1}$ almost surely for every $A$. Thus, for every $A$, $\rho_n(A,1)=\Ex{\edr^{-nS_{\Xc,1}}}=\rho_n(\Xc,1)$ and $$\rho_{j,n-j}(A,A^c,1)=\Ex{\edr^{-nS_{\Xc,1}}}=\rho_{n}(\Xc,1).$$
The  same can be proven, with identical arguments, to hold for any $d$-dimensional correlation sequences $\left\{\rho_{\bm{n}}(\Ac,t)\right\}$  for every collection $\Ac=\{A_1,\ldots,A_d\}$ of pairwise disjoint sets in $\Xcr$.  Since $\rho_n(\Xc,t)=\rho^t_n(\Xc,1)$ for every $t\geq 0$, then $S=\{S(t):=S_{\Xc,t}:\: t\ge 0\}$ is indeed a subordinator, thanks to \eqref{dsub} and $$\rho_{\bm{n}}(\Ac,t)=\Ex{\edr^{-|\bm{n}|S_{t}}},\qquad \bm{n}\in\Z_+^d,$$

This shows just what we wanted to prove: the process $\xi$ with canonical correlations in this form has the same law as a subordinated Dawson-Watanabe process. The form of the generator of $\xi$ comes as a direct application of \cite{Sato99}, Theorem 32.1. 
\end{proof}

An immediate consequence of Theorem \ref{thm:markovcc}, in combination with Theorem \ref{thm:main}, is the following equivalent characterisation.
\begin{cor}
{A time-reversible, homogeneous measure-valued Markov process $\xi=\{\xi_t:\: t\ge 0\}$ with Gamma $(c,P_0)$ reversible measure has canonical autocorrelations if and only if, for every $t\geq o$ the joint Laplace distribution of $(\mu_0,\mu_t)$ is of the form \eqref{eq:kmeasurenew} with underlying random probability kernel given by $$Q_x=\delta_{e^{-S_{t/2}}}\ \ \forall x\in\Xc$$ for a subordinator $\{S_t\,:\:t\geq 0\}$.} 
\end{cor}
\begin{proof} The form of the kernel $Q_x=\delta_{e^{-S_{t/2}}}$ follows from the time-change $t\mapsto S_{t/2}$ applied to $\xi$.
\end{proof}

%

\begin{rmk}
Let $\{\xi_t:\: t\geq 0\}$ be a measure-valued reversible Markov process with canonical autocorrelations and $\Gamma_{cP_0}$ stationary distribution. Theorem \ref{thm:markovcc} and Corollary~\ref{cor:dwalg} imply that there exists a $\R_+$-valued subordinator $S=\{S_t:\: t\geq 0\}$ such that, for every $t\geq 0,$ Algorithm A.4 initialised by the probability distribution of $\edr^{-S_t},$ generates a realisation of $\{(\xi_0(A_i),\xi_t(A_i)):\: i=1,\ldots,d\}$ for every $d$ and every measurable partition $A=(A_1,\ldots,A_d)$ of $\Xc$.
\end{rmk}

We conclude this Section by pointing out another interesting consequence 
 of Theorem~\ref{thm:main}.  Suppose $S$ is a degenerate process such that, for every $t,$ $S_t=2ct$ for some positive constant $2c$. If $\xi$ is a $\Gamma$-DW process with criticality $\lambda=1$, the canonical autocorrelations of $\{\xi^S_t:\: t\ge 0\}$ are determined by powers of $z=\edr^{-c t}$. We have seen in Section \ref{sec:ex1} that this forms the canonical correlation sequence of a re-scaled DW process with criticality parameter $\lambda=2c.$ These are the only instances of Markov canonical autocorrelations directed by a deterministic probability kernel (namely, of the form $Q_x=\delta_{e^{-ct}}$ for all $x$). In all other cases where a non-trivial subordinator $S$ is involved in the time-change, the kernel becomes degenerate at a random constant, i.e.  $Q_x=\delta_{e^{-S(t)}}$ for all $x$, yielding, for every $t$, time-bivariate distributions with a structure as in Section \ref{sec:ex1}. By  the last assertion of Theorem  \ref{thm:main} 
this just proves a new characterisation of the DW process:

\begin{thm}
Let $\{\xi^*_t:\: t\geq 0\}$ be a measure-valued Markov process with $\Gamma_{cP_0}$ stationary distribution and canonical auto-correlations. The pair $(\xi^*_0,\xi^*_t)$ is a bivariate completely random measure, for every $t$, if and only if it is a DW measure-valued branching diffusion with immigration, up to a deterministic re-scaling, namely $\{\xi^*_t:\: t\ge 0\}=\{\lambda \xi_\lambda(t):\: t\ge 0\}$ for some $\lambda\in\R$.
\end{thm}

\end{document}